\documentclass[reqno,a4paper,9pt]{amsart}
\usepackage{marginnote}
\usepackage[utf8]{inputenc} 
\usepackage[english]{babel}
\usepackage[T1]{fontenc}
\usepackage{lmodern}
\usepackage{caption}
\usepackage{subcaption}
\usepackage{enumerate}
\usepackage{amssymb, amsmath}
\usepackage{mathrsfs,dsfont}
\usepackage{amscd}
\usepackage{bbm}
\usepackage{amsthm}
\usepackage{ esint }
\usepackage[mathcal]{euscript}
\usepackage[active]{srcltx}
\usepackage{verbatim}
\usepackage[colorlinks,linkcolor={blue},citecolor={blue},urlcolor={black}]{hyperref}
\usepackage[]{changebar}

\usepackage{xcolor}

\usepackage{mathtools}
\usepackage{hyperref}
\usepackage{url}            
\usepackage{booktabs}       
\usepackage{amsfonts}       
\usepackage{nicefrac}       
\usepackage{microtype}      
\usepackage{lipsum}
\usepackage{fancyhdr}       
\usepackage{graphicx}       
\graphicspath{{media/}}     

\usepackage[numbers]{natbib}

\theoremstyle{plain}
\newtheorem{theorem}{Theorem}[section]
\newtheorem{corollary}[theorem]{Corollary}
\newtheorem{lemma}[theorem]{Lemma}
\newtheorem{proposition}[theorem]{Proposition}

\newtheorem{definition}[theorem]{Definition}

\newtheorem*{definition*}{Definition}

\theoremstyle{remark}
\newtheorem{remark}[theorem]{Remark}
\newtheorem{example}[theorem]{Example}
\newtheorem*{remark*}{Remark}
\newtheorem*{example*}{Example}
\newtheorem*{notation*}{Notation}

\numberwithin{equation}{section}

\def\N{{\mathbb N}}
\def\Z{{\mathbb Z}}
\def\R{{\mathbb R}}

\def\W{{\mathcal W}}

\def\norma #1{\left\lVert #1 \right\rVert}
\def\nm #1{ \left\langle #1 \right\rangle}

\def\P{{\mathscr{P}}}
\def\ev{{\rm ev}}

\def\E{{\mathbb E}}
\def\F{{\mathscr F}}
\def\B{{\mathscr B}}

\newcommand{\cP}{\mathbb{P}}

\def\cE{{\mathscr E}}

\def\Y{{\mathcal Y}}
\def\T{{\mathcal T}}
\def\y{{\boldsymbol y}}
\def\v{{\boldsymbol v}}
\def\bT{{\boldsymbol \T}}
\def\bB{{\boldsymbol B}}
\def\X{{\boldsymbol X}}
\def\by{{\boldsymbol Y}}
\def\blambda{{\boldsymbol{\lambda}}}
\newcommand{\hol}[1]{\"{#1}}
\def\w2{\stackrel{2}{\rightharpoonup}}

\def\strong2{\stackrel{2}{\rightarrow}}

\definecolor{viola}{rgb}{0.3,0,0.7}
\definecolor{ciclamino}{rgb}{0.5,0,0.5}
\definecolor{rosso}{rgb}{0.8,0,0}

\newcommand{\beq}{\begin{equation}}
\newcommand{\eeq}{\end{equation}}
\newcommand{\bal}{\begin{aligned}}
\newcommand{\eal}{\end{aligned}}
\newcommand{\ben}{\begin{enumerate}}
\newcommand{\beni} {\begin{enumerate}[(i)]}
\newcommand{\een}{\end{enumerate}}
\newcommand{\bit}{\begin{itemize}}
\newcommand{\eit}{\end{itemize}}
\newcommand{\beqw}{\begin{equation*}}
\newcommand{\eeqw}{\end{equation*}}
\newcommand{\bex}{\begin{example}}
\newcommand{\eex}{\end{example}}
\newcommand{\bre}{\begin{example}}
\newcommand{\ere}{\end{example}}
\newcommand{\bma}{\begin{bmatrix}}
\newcommand{\ema}{\end{bmatrix}}

\renewcommand{\hat}{\widehat}

\newcommand{\cY}{\mathcal{Y}}

\newcommand{\de}{\mathrm{d}}

\newcommand{\M}{\mathcal{M}}

\renewcommand{\tilde}{\widetilde}

\makeatletter
\@namedef{subjclassname@2020}{%
  \textup{2020} Mathematics Subject Classification}
\makeatother

\title[A large multi-agent system with noise both in position and control]{
A large multi-agent system with noise both in position and control}

\author[G. D'Onofrio]{Giuseppe D'Onofrio}
\address[Giuseppe D'Onofrio]{Dipartimento di Scienze Matematiche ``G.~L.~Lagrange'',
Politecnico di Torino, Corso Duca degli Abruzzi 24,
10129 Torino, Italy.}
\email{giuseppe.donofrio@polito.it}

\author[A. Melchor Hernandez]{Anderson Melchor Hernandez}
\address[A. Melchor Hernandez]{Dipartimento di Matematica,
Università di Bologna, Via Zamboni 33,
40126, Bologna, Italy.}
\email{anderson.melchor@unibo.it}
\date{\today}

\begin{document}

\begin{abstract}
In this work, we consider a multi-population system where the dynamics of each agent evolve according to a system of stochastic differential equations in 
a general functional setup, determined by the global state of the system. Each agent is associated with a probability measure, that assigns the label accounting for the population to which the agent belongs.
We do not assume any prior knowledge of the label of a single agent, and we allow that it can change as a consequence of the interaction among the agents.
Furthermore, the system is affected by noise both in the agent's position and labels.
First, we study the well-posedness of such a system and then a mean-field limit, as the number of agents diverges, is investigated together with the analysis of the properties of the limit distribution both with  Eulerian and Lagrangian perspectives.
As an application, we consider a large network of interacting neurons with random synaptic weights, introducing resets in the dynamics.
\end{abstract}

\maketitle

\tableofcontents

\section{Introduction}
Mean-field analysis stands as a formidable mathematical tool employed to investigate large-scale systems characterized by a multitude of interacting components. These components may take the form of particles, agents, or individuals, and the methodology behind mean-field analysis allows the study of emergent behaviors at a collective level. Over the years, mean-field analysis has evolved into a versatile tool, finding applications across a diverse range of scientific disciplines. In physics, it has been instrumental in understanding phase transitions and collective phenomena in self-driven particle systems, as exemplified by the work of Vicsek et al. \cite{MR3363421}. In the domain of economics, mean-field game theory, as developed by Lasry and Lions \cite{MR2295621}, provides a framework for analyzing strategic interactions among a large number of agents. Moreover, the insights garnered from mean-field analysis have proven invaluable in the field of neuroscience, enabling the study of emergent properties arising from the complex interactions of a large number of neurons (see for instance \cite{delarue2015particle, duchet2023mean, dumont2020mean, robert2016dynamics, treves} just to name a few). 
More recently, these tools are being used in the context of machine learning and artificial neural networks (\cite{han2019mean, sirignano2022mean}). In economics, it is customary to refer to mean field games, where the term denotes the analysis of the average behavior of a large number of players or agents. The evolution of such models is typically studied using the so-called Hamilton–Jacobi–Bellman equations, often coupled with a Fokker–Planck type equation \cite{cardaliaguet2017learning,pham2018bellman}.
In all these models, the interacting agents or individuals are assumed to belong to different species without explicit spatial dependency. In most cases, this assumption arises due to the low abundances of agents or individuals, which can disrupt the convergence towards a limit model or introduce stochastic properties in the final model, as discussed in \cite[Remark 2.7]{Lea2023}. 

In this work, we consider a stochastic population model where an agent located at position $x\in \R^{d}$ is associated with a probability measure $\lambda$ supported in some compact metric space $(U,d_{U})$. Usually, we represent $U$ as a set of labels accounting for the population to which the agent belongs. In particular, to describe the time evolution of the system, we consider for each agent $i$ a stochastic differential equation (SDE) of the following form:
 
\begin{align}\label{introeq1}
    \de X^{i}(t)=v_{\Lambda_{t}^{N}}(X^{i}(t),\lambda^{i}(t))\de t+ \sqrt{2\sigma}\de B^{i}(t)
\end{align}
where the position $X^{i}$ at time $t$  is driven by a velocity field $v_{\Lambda_{t}^{N}}$, $t\in [0,T]$, where $\Lambda_{t}^{N}$ is the empirical measure $\frac{1}{N}\sum_{i=1}^{N}\delta_{(X^{i}(t),\lambda^{i}(t))}$ and $\lambda^{i}$ is the probability distribution of labels for the agent $i$. Furthermore, the dynamics of $\lambda^{i}$  are given by the differential equation 
\begin{align}\label{introeq2}
\de\lambda^{i}(t)=\T_{\Lambda_{t}^{N}}(X^{i}(t),\lambda^{i}(t))\de t,    
\end{align}
where $\T_{\Lambda_{t}^{N}}$ is a driving operator (the precise setting of the involved functional spaces is given in Section \ref{section2}). The term $\lambda^i$ expresses the probability that the agent $i$, at a certain moment in time, has a certain label, i.e. belongs to a subset of $U$. Indeed, we do not assume any prior knowledge of the label of a single agent, and we allow that it can change as a consequence of the interaction among the agents, like in \cite{MS2020}. Recently, in \cite{almi2023opt}, the authors studied a model similar to \eqref{introeq1}. However, in their study, $\lambda^{i}(t)$ is given as the solution of non-local continuity equation. In such a situation, the pair $({\rm law}(X^{i}(t)), \lambda^{i}(t))$ represents a system of agents divided into two categories. The probability law of $X^{i}(t)$ represents an uncontrolled population called followers, while $\lambda^{i}(t)$ represents the controlled population called leaders. It is important to notice that in our case, the effect of the noise in the variable $X^{i}(t)$ affects the behavior of $\lambda^{i}(t)$ by means of the velocity field $\T$, and thus we need to study the well-posedness of the system given by \eqref{introeq1}-\eqref{introeq2}. For a similar model without diffusion, we refer to \cite{fcarlo2019}.
The idea of a multi-agent model, where individuals are equipped with strategies and influenced by diffusive effects, was previously explored in  \cite{baldi2023buona} (see also \cite{my_unpublished_work}) as an extension of \cite{AFMS}. Here, we develop the study more in the spirit of \cite{MS2020}.
In model \eqref{introeq1}-\eqref{introeq2}, we consider the sequence $(B^{i}(t))_{i\in \N}$, $t\in [0,T]$, as a set of independent Brownian motions. The characteristic feature of this model is that, in the limit, we obtain a continuity equation where a diffusive part is developed. Here, we assume $\sigma>0$, as the case $\sigma=0$ has already been studied in some previous works, such as \cite{AFMS,Ascione-Castorina-Solombrino,MS2020}. Furthermore, for the sake of generality, we are mainly concerned with local regularity properties of both $v_{\Lambda_{t}^{N}}, \T_{\Lambda_{t}^{N}}$, that is, both fields do not admit a global Lipschitz constant, and thus we need to enforce our hypotheses by considering a further sublinearity property of the fields in the involved arguments. 

In the mean-field limit, we study also the time-regularity of the limit of the sequence of the empirical measures  as a solution of an evolution equation. In doing that, we make use of the concept of Eulerian and Lagrangian solutions.

The model \eqref{introeq1}-\eqref{introeq2} straightforwardly can be used for various applications just reinterpreting the role of the agents and the controls (labels). 
As an applications of the previous model
we consider a system of interacting neurons whose synaptic weights, that describe the influence that a neuron  has onto the others, follow a certain distribution in the spirit of neuronal networks with random synaptic weights \cite{delarue2019, faugeras2020asymptotic, faugeras2009constructive}.
In this context, that belongs to the class of noisy integrate-and-fire models, the stochastic process $X^{i}$ describes the time evolution of the membrane potential of a neuron receiving the inputs from the other neurons in the network it is embedded in. These contributions are summarized in the drift part of \eqref{introeq1} according to the velocity field $v_{\Lambda_{s}^{N}}$ until the membrane potential reaches a physiological value $X_F$ triggering a spike.
As a result of this spike, which is assumed to occur instantaneously, the membrane potential of all the other connected neurons receives a contribution. Afterward, the dynamics of the spiking neuron restart from the resting state $X_R$.
The strength of the interaction/contribution between neurons and its evolution is described through a general operator $\T_{\Lambda_{t}^{N}}$ that 
is perturbed by a random effect of the environment, being it dependent on $X^{i}$.

It is important to note that  model \eqref{introeq1} assumes that the global dynamics of $X^{i}$ does not exhibit resets. Therefore, for neuronal modeling, here we include the feature of the spike generation; that is, we consider the system

\begin{align}\label{intro:synap1}
\begin{dcases}
&X^{i}(t)=X_{0}^{i}+ \int_{0}^{t}v_{\Lambda_{s}^{N}}(X^{i}(s),\lambda^{i}(s))\de s+ \sqrt{2\sigma}B_{t}^{i},  \hskip 0,1cm \text{if $X^{i}(t^{-})<X_{F}$,}\\
&X^{i}(t)=X_{R}, \phantom{+\int_{0}^{t}v_{\Lambda_{s}^{N}}(X^{i}(s),\lambda^{i}(s))\de s+ \sqrt{2\sigma}B_{t}^{i}}\hskip 0,2cm \text{if $X^{i}(t^-)=X_{F}$,}\\
&\lambda^{i}(t)=\lambda_{0}^{i}+\int_{0}^{t}\T_{\Lambda_{s}^{N}}(X^{i}(s),\lambda^{i}(s))\de s,
\end{dcases}
\end{align}
where $X_F$ is an additional parameter describing the spiking threshold and $X_R$ is the reset value (resting potential).

Here, we are particularly interested in the following case. Let us consider a  measurable function $\alpha:U\times \Y\rightarrow \R$, and define $\T:\Y\times \P_{1}(\Y) \rightarrow \F(U)$ as follows:
\begin{align*}
\begin{aligned}
&\T(y,\Psi)\coloneqq \int_{\Y}\alpha(\cdot,z)\de \Psi(z), \hskip 0,1cm \text{ such that} \hskip 0,1cm \int_{U}\int_{\Y}\alpha(u,z)\de \Psi(z)\de\mu(u)=0,\\
&\mu(u)=\sum_{k\geq 1}\beta_{k}\delta_{u_{k}}, \hskip0,1cm u_{k}\in U, \hskip 0,1cm \beta_{k}\in \R,
\end{aligned}
\end{align*}
where $\mu\in \M(U)$ defines a signed measure. Furthermore, we consider a Lipschitz velocity field $b: \R^{d}\rightarrow \R^{d}$, and set $v_{\T}:\Y\rightarrow \R^{d}$, $(x,\lambda)\mapsto b(x)$. 
The functional form of $\T$
 is of great generality and contains, as special cases, the random synaptic weights seen as constants, Bernoulli or Gaussian random variables  considered in previous works \cite{de2015hydrodynamic, delarue2019, eva_toymodel, faugeras2009constructive}.

Moreover, since weight distributions in real cortex have been observed to be broad, neurons can belong to different subgroups with different behaviors. The variance of the temporal fluctuations of the input currents may be large enough to endow the network
with its own source of variability (see \cite{la2021mean} for an exhaustive review). Therefore, in Section \ref{sec:evol1}, we consider a heterogeneous population of neurons. In mathematical terms, this translates in an additional source of noise in equation \eqref{introeq2}, namely
$$\lambda^{i}(t)=\lambda_{0}^{i}+\int_{0}^{t}\T_{\Lambda_{s}^{N}}\left(X^{i}(s),\lambda^{i}(s)+ R(s)\right)\de s,$$
with $R(t)\coloneqq \sum_{h\in \N}a_{h}W_{h}(t)e_{h}$,
where $(W_{h}(t))_{h\in \N}$ is a sequence of independent Brownian motions, $(a_{h})_{h\in\N}$  is a sequence of non-negative numbers such that $\sum_{h}a_{h}^{2}<+\infty$, and $(e_{h})_{h}$ is a proper sequence of signed measures.

For all these models, we prove the well-posedness of the coupled SDEs and discuss the link between Eulerian and Lagrangian solutions of \eqref{introeq1}-\eqref{introeq2}. Finally, for \eqref{intro:synap1} we show the qualitative behavior in the one-dimensional case using simulations.

Classical models of large neuronal network can be studied also using a diffusion approximation, which, however, requires that the frequency of the inputs is high, but their amplitude is small (see \cite{tamborrino2021shot}).
Mean-field theory is a device to analyze the collective behavior of such a dynamical system
without these requests as long as the number of interacting particles is high. The
theory allows to reduce the behavior of
the system to the properties of a handful
of parameters. In neural circuits, these
parameters are typically the firing rates. In this context, we are often interested
in the properties averaged across the distribution
of possible weights  (\cite{ascione2020optimal,baxendale2011sustained,ditlevsen2013morris}) and in the limit of infinite
network size \cite{corm, locherbach2024}. Some important properties, such as
the existence of a sharp phase transition or synchronization, are only
obtainable in this limit.
However, fluctuations are sometimes so high that care must be taken in the averaging process, as in the case of multiplicative noise; see \cite{ascione2023deterministic}.
For this reason we introduce sources of noise in the network itself. This goes also in the direction of addressing the modeling of synaptic plasticity, where the synaptic weight  itself is a
random process evolving in time, as suggested in \cite{locherbach2017spiking}.

The paper is organized as follows: in Section \ref{section2} we introduce our notation and we recall some preliminary  results. 
We describe the multi-agent system described by coupled SDEs with noise both in position and labels along with the relative assumptions and a description of the spaces in which we are working. 
The main results on the well-posedness of the considered system, along with comments on the properties of the mean-field limit, are stated and proven in Section \ref{section3}. In particular, we develop the mean-field description using both the Eulerian and Lagrangian approaches to study the time regularity of the limit distribution. 
Finally, in Section \ref{section4}, we consider a large network of interacting neurons whose activity is described by the aforementioned system, both for homogeneous and heterogeneous populations of neurons. This model exhibits jumps, and the corresponding well-posedness results are provided.

\section{The model}\label{section2}
In this section, we will provide a detailed overview of the model, including a comprehensive description of all the tools used and the key assumptions it is based on.
 In details, subsection \ref{subsec1}  introduces some preliminaries and notation, while in subsection \ref{subsec2} we provide the functional spaces needed for our purpose, and in subsection \ref{subsec3} our assumptions. Furthermore, in subsection \ref{subsec4}, we recall the notion of Eulerian and Lagrangian solutions. Subsection \ref{subsec5} is dedicated to the particle system and to the notion of strong solution for SDEs.

\subsection{Preliminaries and notation}\label{subsec1}

Let $(X,d_{X})$ be a metric space, we denote by $\M(X)$ the space of signed Borel measures with finite total variation, by $\M_{+}(X)$ and $\P(X)$ the convex subsets of nonnegative measures and probability measures, respectively. For $\sigma\in \M(X)$, $\vert \sigma\vert\in \M_{+}(X)$ denotes the total variation measure of $\sigma$. We shall also use the notation $\M_{0}(X)$ for the subset of measures with zero mean. Given a metric space $(X,d_{X})$, we consider the Lipschitz space 
\begin{align*}
{\rm Lip}(X,d_{X}):=\left\{\phi: X\rightarrow \R\vert \exists\hskip 0,1cm L>0: \forall x,y\in X\hskip 0,1cm \vert \phi(x)-\phi(y)\vert \leq L d_{X}(x,y)\right\}.
\end{align*}
For a continuous function $\phi\in C(X)$ we denote by
\begin{align*}
{\rm Lip}(\phi):=\sup_{\stackrel{x,y\in Y}{x\neq y}}\frac{\vert \phi(x)-\phi(y) \vert}{d_{X}(x,y)}   \end{align*}
its Lipschitz constant. In the next, we endow ${\rm Lip}(X,d_{X})$ with the Lipschitz norm $\norma{\phi}_{{\rm Lip}}:=\sup_{x\in X}\vert \phi(x)\vert+ {\rm Lip}(\phi)$. 
Given a measure $\mu\in \M_{+}(X)$, $(E,d_{E})$ another metric space, and $f:X\rightarrow E$ a $\mu$-measurable function, we shall denote by $f_{\#}\in \M_{+}(E)$ the push-forward measure, having the same mass as $\mu$ and defined by $f_{\#}\mu(B)=\mu(f^{-1}(B))$ for any Borel set $B\subset E$. We often use the change of variable formula

\begin{align*}
\int_{E}g {\rm d}f_{\#}\mu= \int_{X}g\circ f {\rm d}\mu
\end{align*}
whenever either one of the integrals makes sense. In a complete and separable metric space $(X,d_{X})$, we shall use the Kantorovich-Rubinstein distance $\W_{1}(\cdot,\cdot)$ in the class $\P(X)$. For $\mu_{1},\mu_{2}\in \M_{1}(X)$, the $1$-Wasserstein distance $\W_{1}(\mu_{1},\mu_{2})$ is defined by 

\begin{align*}
\W_{1}(\mu_{1},\mu_{2}):=\inf\left\{\left. \int_{X\times X}d_{X}(x_{1},x_{2}) \gamma(dx_{1},dx_{2})\right\vert 
 \gamma\in \Gamma(\mu_{1},\mu_{2})\right\}   
\end{align*}
where $\Gamma(\mu_{1},\mu_{2})$ is the set of admissible coupling between $\mu_{1}$ and $\mu_{2}$. It is worth to recall that due to Kantorovich duality, one can also consider the following definition 

\begin{align*}
\W_{1}(\mu_{1},\mu_{2}):=\sup\left\{ \int_{X}\phi{\rm d}(\mu_{1}-\mu_{2})\Big| \hskip 0,1cm \phi \in {\rm Lip}(X,d_{X}), {\rm Lip}(\phi)\leq 1 \right\}.  
\end{align*}

Notice that $\W_{1}(\mu_{1},\nu_{1})$ is finite if $\mu_{1},\nu_{1}$ belong to the space

\begin{align*}
\P_{1}(X):=\left\{ \mu\in \P(X)\Big| \int_{X}d_{X}(x,\overline{x}){\rm d}\mu(x)<\infty \hskip 0,2cm \text{for some $\overline{x}\in X$}\right\}.
\end{align*}

Note that $(\P_{1}(X),\W_{1})$  is complete if $(X,d_{X})$ is complete. Moreover, by \cite[Theorem 2.2.1]{panaretos2020} the following holds true: a sequence $(\mu_{n})\subset \P_{1}(X)$ converges to $\mu\in \P_{1}(X)$ with respect to the Wasserstein distance $\W_{1}$ if and only if,
for all $\phi\in {\rm Lip}(X,d_{X})$,

\begin{align*}
&\int_{X}\phi {\rm d}\mu_{n} \overset{n\rightarrow \infty}{\longrightarrow}  \int_{X} \phi{\rm d}\mu, \qquad  \int_{X}{\rm d}_{X}(\cdot,\bar{x}){\rm d}\mu_{n}\overset{n\rightarrow \infty}{\longrightarrow}  \int_{X}d_{X}(\cdot,\bar{x}){\rm d}\mu.  
\end{align*}
 Suppose that $E$ is a Banach space, the notation $C_{b}^{1}(E)$ will be used to denote the subspace of $C_{b}(E)$ of functions having bounded continuous Fréchet differential at every point. Hence, we make use of the notation $D\varphi$ to denote such a  Fréchet differential.  Let $T>0$, in the case $\varphi:[0,T]\times E\rightarrow \R$, we make use of the symbol $\frac{\de}{\de\,t}$ to denote the partial derivative with respect to the variable $t\in[0,T]$, and $D$ will  only stand for the derivative with respect to the variables in $E$.

\subsection{The framework}\label{subsec2}

In what follows, we are interested in studying SDEs on $\Y$, where $\Y\coloneqq \R^{d}\times \P(U)$. The state space of our system is given by pairs $(x,\lambda)=y\in\Y$. Here $x\in\R^{d}$ denotes the spatial component of an agent, whereas the element $\lambda\in \P(U)$ denotes a probability distribution over the space $U$, which we assume to be a compact metric space. It can be interpreted as a space of strategies \cite{AFMS,MS2020}. 

We now consider a compact metric space $(U,d_{U})$, and consider the set $\F(U)$ defined as

 \begin{align*}
 \F(U)\coloneqq\overline{{\rm span}(\P(U))}^{{\norma{\cdot}}_{{\rm BL}}}   
 \end{align*}
as the closure in the dual space $({\rm Lip}(U,d_{U}))^{\ast}$ with respect to the dual norm

\begin{align*}
{\norma{\ell}}_{{\rm BL}}:=\sup\left\{\nm{\ell,\phi}:\phi\in {\rm Lip}(U, d_{U}), \norma{\phi}_{{\rm Lip}}\leq 1\right\}.
\end{align*}
The space $\F(U)$ is the so-called  Arens-Eells space introduced in \cite{Arens-Eells}. This is 
 a separable Banach space containing $\M(U)$. 
Furthermore,  for a measure $\nu\in \M_{0}$, the $\norma{\cdot}_{{\rm BL}}$-norm is equivalent to the norm induced by the dual formulation of the $1$-Wasserstein distance.

On the other hand, note that $\M(U)$, when endowed with the  total variation norm

\begin{align*}
    \norma{\sigma}_{{\rm TV}}:=\sup\left\{\int_{U}\phi d\sigma: \phi\in C(U), \vert \phi\vert\leq 1\right\}
\end{align*}
has the structure of Banach space, isometrically isomorphic to the dual of $C(U)$. We will also use the representation formulas

\begin{align*}
\norma{\sigma}_{{\rm TV}}=\vert \sigma\vert(U)=\sup\left\{\int_{U}\phi d\sigma: \phi\in {\rm B}_{b}(U), \vert \phi\vert\leq 1\right\}
\end{align*}
where $B_{b}(U)$ denotes the class of bounded Borel functions $\phi: U\rightarrow \R$. Furthermore,  we have by Kantorovich duality that

\begin{align*}
   \norma{\mu_{1}-\mu_{2}}_{{\rm BL}}\leq \W_{1}(\mu_{1},\mu_{2})\leq (1+ D_{U})\norma{\mu_{1}-\mu_{2}}_{{\rm BL}}.
\end{align*}

In the next we need to consider the space $\cE\coloneqq\R^{d}\times \F(U)$ endowed with the norm $\norma{y}_{\cE}=\norma{(x,\sigma)}_{\cE}:=\vert x\vert + \norma{\sigma}_{{\rm BL}}$, which is a separable Banach space. For $y\in \Y$ and $\Psi\in \P_{1}(\Y)$, we consider the vector field $b_{\Psi}:\Y\rightarrow \cE$ defined as

\begin{align}\label{vectorf}
b_{\Psi}(y)\coloneqq 
\begin{pmatrix}
&\hskip -0,3cm v_{\Psi}(y)\\
&\hskip -0,3cm \T_{\Psi}(y)
\end{pmatrix}.
\end{align}
The first component of $b_{\Psi}$ is a velocity field in $\R^{d}$ determined by the global state of the system $\Psi$; the second component is expressed in terms of an operator $\T_{\Psi}:\Y\rightarrow \cE$ which sees the location of agents around $\Psi$.

 For a given $R>0$, we denote by $B_{R}$ the closed ball of radius $R$ in $\R^{d}$ and by $B_{R}^{\Y}$ the ball of radius $R$ in $\Y$, that is, $B_{R}^{\Y}\coloneqq \{y\in\Y:\norma{y}_{\Y}\leq R\}$. 
\begin{remark}\label{compacty}
Since $U$ is assumed compact, by \cite[Corollary 2.2.5]{panaretos2020} $\P(U)$ is compact and thus $\Y$ is a locally compact space. Therefore $B_{R}^{\Y}$ is a compact set. Since $\Y\subset \cE$, for any $\Psi\in \P(\Y)$ we may define
\end{remark}

\vspace{-8mm}
\begin{align*}
m_{1}(\Psi)\coloneqq \int_{\Y}\norma{y}_{\cE}\de \Psi.
\end{align*}
Hence, 
\begin{align*}
\P_{1}(\Y)=\{\Psi\in \P(\Y): m_{1}(\Psi)<+\infty\}.
\end{align*}

\subsection{Assumptions}\label{subsec3}

We assume that the velocity field $v_{\Psi}:\Y\rightarrow \R^{d}$ satisfies the following structural assumptions:

\begin{enumerate}[({A}1)]
\item 	\label{A1} For every $R>0$, for every $\Psi\in \P(B_{R}^{\Y})$, $v_{\Psi}\in {\rm Lip}(B_{R}^{\Y};\R^{d})$ uniformly with respect to $\Psi$, that is, there exists a positive constant $L_{v,R}$ such that
\begin{align*}
\vert v_{\Psi}(y^{1})-v_{\Psi}(y^{2})\vert \leq L_{v,R}\norma{y^{1}-y^{2}}_{\cE};
\end{align*}
\item for every $R>0$, for every $\Psi\in\P(B_{R}^{\Y})$, there exists a positive constant $L_{v,R}$ such that for every $y\in B_{R}^{\Y}$ and for every $\Psi^{1},\Psi^{2}\in \P(B_{R}^{\Y})$
\begin{align*}
\vert v_{\Psi^{1}}(y)-v_{\Psi^{2}}(y)\vert \leq L_{v,R}\W_{1}(\Psi^{1},\Psi^{2});
\end{align*}
\item\label{A3} there exists $M_{v}>0$ such that for every $y\in\Y$ and for every $\Psi\in\P_{1}(\Y)$ there holds 
\begin{align*}
\vert v_{\Psi}(y)\vert\leq M_{v}\left(1+\norma{y}_{\cE}+ m_{1}(\Psi)\right). 
\end{align*}
\end{enumerate}
The assumptions on the operator $\T$ are the following. For $\Psi\in \P_{1}(\Y)$, the operator $\T_{\Psi}:\Y\rightarrow \F(U)$ is  such that

\begin{enumerate}[({B}1)]
\item\label{B1} for every $(y,\Psi)\in \Y\times P_{1}(\Y)$, constants belong to the kernel of $\T_{\Psi}(y)$, that is
\begin{align*}
\nm{\T_{\Psi}(y),1}_{\F(U),{\rm Lip}(U)}=0;
\end{align*}
\item\label{B2}for every $(y,\Psi)\in \Y\times \P_{1}(\Y)$, there exists a positive constant $M_{\T}$ such that
\begin{align*}
\norma{\T_{\Psi}(y)}_{{\rm BL}}\leq M_{\T}(1+\vert x\vert+m_{1}(\Psi)), \hskip0,1cm\text{where $y=(x,\lambda)$;}
\end{align*}
\item  for every $R>0$ there exists  a positive constant $L_{\T,R}$ such that for every $(y^{1},\Psi^{1}),(y^{2},\Psi^{2})\in B_{R}^{\Y}\times \P(B_{R}^{\Y})$,
\begin{align*}
\norma{\T_{\Psi^{1}}(y^{1})-\T_{\Psi^{2}}(y^{2})}_{{\rm BL}}\leq L_{\T,R}\left(\norma{y^{1}-y^{2}}_{\cE}+ \W_{1}(\Psi^{1},\Psi^{2})\right);
\end{align*}
\item\label{B4} for every $R>0$ there exists $\delta_{R}>0$ such that for every $(y,\Psi)\in B_{R}^{\Y}\times \P_{1}(\Y)$ we have

\begin{align*}
\T_{\Psi}(y)+\delta_{R}\lambda\geq 0,  \hskip0,1cm\text{where $y=(x,\lambda)$}.
\end{align*}
\end{enumerate}

\begin{proposition}[{\cite[Proposition 3.1]{MS2020}}]\label{morasolo1}
For $y\in\Y$ and $\Psi\in\P_{1}(\Y)$, define $b_{\Psi}(y)$ as in \eqref{vectorf}. Assume that $v_{\Psi}:\Y\rightarrow\R^{d}$ satisfies (A\ref{A1})-(A\ref{A3}) and $\T_{\Psi}:\Y\rightarrow \F(U)$ satisfies (B\ref{B1})-(B\ref{B4}). Then the following hold true:
\begin{enumerate}
\item[(i)] for every $R>0$, for every $\Psi\in\P(B_{R}^{\Y}),$ and for every $y^{1},y^{2}\in B_{R}^{\Y}$, there exists a positive constant $L_{R}$ such that
\begin{align*}
\norma{b_{\Psi}(y^{1})-b_{\Psi}(y^{2})}_{\cE}\leq L_{R}\norma{y^{1}-y^{2}}_{\cE};
\end{align*}
\item[(ii)] for every $R>0$; for every $\Psi^{1},\Psi^{2}\in \P(B_{R}^{\Y})$, and for every $y\in B_{R}^{\Y}$, there exists a positive constant $L_{R}$ such that
\begin{align*}
\norma{b_{\Psi^{1}}(y)-b_{\Psi^{2}}(y)}_{\cE}\leq L_{R}\W_{1}(\Psi^{1},\Psi^{2});
\end{align*}
\item[(iii)]for every $R>0$, there exists $\theta>0$ such that for every $y\in B_{R}^{\Y}$ and for every $\Psi\in \P(B_{R}^{\Y})$

\begin{align*}
y+\theta b_{\Psi}(y)\in \Y;
\end{align*}
\item[(iv)]there  exists $M>0$ such that for every $y\in\Y$ and for every $\Psi\in\P_{1}(\Y)$ there holds 

\begin{align*}
\norma{b_{\Psi}(y)}_{\cE}\leq M(1+\norma{y}_{\cE}+ m_{1}(\Psi)).
\end{align*}
\end{enumerate}
\end{proposition}

\subsection{Eulerian and Lagrangian solutions}\label{subsec4}
Let us discuss the concepts of Eulerian and Lagrangian solutions in the context of studying the asymptotic behavior of a system like \eqref{introeq1}-\eqref{introeq2} as $N\rightarrow +\infty$. Let us first recall the notion of Eulerian solution. In the next, given a  sufficiently smooth function $\varphi:\Y\rightarrow \R$, and  generic $\Psi\in \P_{1}(\Y)$ we define the operator 

\begin{align}\label{diffop}
{\mathcal L}_{\Psi}\varphi(y)\coloneqq \sigma{\rm Tr}({\rm D}_{yy}\varphi(y))+ \nm{b_{\Psi},{\rm D}_{y}\varphi(y)}  
\end{align}
where $b_{\Psi}$ is defined according to \eqref{vectorf}. In the next, we make the following important assumption about ${\mathcal L}_{\Psi}$. 

\begin{enumerate}
\item[(C$1$)] Let us suppose that there exists $r>0$ such that for each $\Psi\in \P_{1}(B_{r}^{\Y})$, ${\mathcal L}_{\Psi}:{\rm Dom}({\mathcal L}_{\Psi})\subset C_{0}(\Y)\rightarrow C_{0}(\Y)$ is densely defined, a dissipative linear operator such that for some $\eta>0$
\begin{align*}
    {\rm Range}(\,\eta I- {\mathcal L}_{\Psi})= C_{0}(\Y).
\end{align*}
Here, we denote by $C_{0}(\Y)$ the space of continuous real-valued functions on $\Y$ vanishing at infinity. That is, $\varphi\in C_{0}(\Y)$ if it is continuous and for all $\varepsilon>0$ there exists a compact set $K_{\varepsilon}\subset \Y$ such that $\vert \varphi(y)\vert<\varepsilon$ for all $y\notin K_{\varepsilon}$.
\end{enumerate}

\begin{definition}\label{dfn:Eulero}
Let $\Lambda\in C^{0}([0,T];(\P_{1}(\Y),\W_{1}))$ and let $\bar{\Lambda}\in \P(\Y)$ be a given compactly supported initial condition. We say that $\Lambda$ is an Eulerian solution to the parabolic equation

\begin{align}\label{e:parabolic}
\frac{\de}{\de t}\Lambda_{t}+\mathcal L_{\Lambda_{t}}^{\ast}(\Lambda_{t})=0,
\end{align}
 starting from $\bar{\Lambda}$ if and only if for every $\varphi\in C_{b}^{1}([0,T]\times\overline{\Y})$,
\begin{align}\label{phit}
\int_{\Y}\varphi(t,y)\de\Lambda_{t}(y)&-\int_{\Y}\varphi(0,y)\de\bar{\Lambda}(y)
=\int_{0}^{t}\int_{\Y}\left( \frac{\de}{\de t}\varphi(s,y)+ \mathcal L_{\Lambda_{s}}\varphi(s,y)\right)\de\Lambda_{s}(y)\de s,
\end{align}
where $\mathcal L_{\Lambda_{t}}^{\ast}$ denotes the dual operator of $\mathcal L_{\Lambda_{t}}$. Notice that for us $\varphi$ will be always a function which does not depend on $t$, and then its time derivative is not needed in \eqref{phit}. Notice that our operator ${\mathcal L}_{\Psi}$ can be related to an SDE. 
\end{definition}

\begin{remark}\label{itoderivation}
Suppose that we have a generic Brownian motion $(\overline{B}(t))_{t\in[0,T]}$ defined in a complete probability space $(\Omega,\{\mathcal{F}_{t}\}_{t\in [0,T]},\cP)$. Let us consider the equations
\begin{align}\label{eq:central}
\begin{aligned}
&X(t)=\bar{X}_{0}+ \int_{0}^{t} v_{\Lambda}(X(s),\lambda(s))\de s + \sqrt{2\sigma}\bar{B}(t),\\
&\lambda(t)=\bar{\lambda}_{0}+ \int_{0}^{t}\T_{\Lambda}(X(s),\lambda(s))\de s,
\end{aligned}
\end{align}
where $X(0)=\bar{X}_{0}\in L^{2}(\Omega;\R^{d})$, $\lambda_{0}\in L^{2}(\Omega,\P(U))$ are the initial conditions for \eqref{eq:central}. Since $\cE$ is a separable Banach space, thanks to It\hol{o} calculus, we can express ${\mathcal L}_{\Psi}$ in a more specific manner. Indeed, suppose that $\varphi:\cY \rightarrow \R$ is a $C^{2}(\Y)$ function. Then by the It\hol{o} formula for $ \de \varphi(X(t),\lambda(t))$ one has:

 \begin{align*}
  \de \varphi(X(t),\lambda(t))=\nabla_{x}\varphi(X(t),\lambda(t))\de X(t)+ D_{\lambda}\varphi(X(t),\lambda(t))\de\lambda(t)+ \sigma \Delta_{x}\varphi(X(t),\lambda(t))\de t  
 \end{align*}
 where $D_{\lambda}$ denote the Frechét derivative with respect to  the variable $\lambda$, and $(X(t),\lambda(t))$ is the solution of \eqref{eq:central} above. By taking the integral with respect to the time variable, and the expected valued with respect to $\cP$, we obtain that

 \begin{align*}
 \begin{aligned}
 \int_{\Y}\varphi(y)\de \mu_{t}(y)- \int_{\Y}\varphi(y)\de \mu_{0}(y)=
 &\int_{0}^{t}\int_{\Y}\nm{D_{y}\varphi(y),b_{\Psi}(y)}\mu_{s}(y)\de s +\\
 &\hskip 0,5cm+\sigma\int_{0}^{t}\int_{\Y}\Delta_{x}\varphi(y)\de\mu_{s}(s)\de s.
 \end{aligned}
 \end{align*}
 Hence, we obtain 

 \begin{align*}
{\mathcal L}_{\Psi}\varphi(y)=\sigma\Delta_{x}\varphi(y)+\nm{b_{\Psi},D_{y}\varphi(y)}.  
 \end{align*}
Our assumption C1 guarantees the existence of a strongly continuous semigroup of contractions on $C_{0}(\Y)$. In general, this can be proven when the first differential equation of \eqref{eq:central}, has the following structure

\begin{align*}
 &\de X(t)=(AX(t)+v_{\Lambda}(X(t),\lambda(t)))\de t + \sqrt{2\sigma}\de \bar{B}(t),   
\end{align*}
where $A:\R^{d}\rightarrow \R^{d}$ is the infinitesimal generator of a strongly continuous semigroup. However, in our case $A=0$, and thus results such as \cite[Proposition 2.8]{ma2012introduction}, which establish that right processes and strongly continuous semigroups with generator ${\mathcal L}_{\Psi}$ are associated in the sense of Dirichlet forms, cannot be applied. Moreover, results such as \cite[Theorem 7.7]{da2014}, which establish the existence of solutions to \eqref{eq:central}, cannot be applied either.
\end{remark}

Before  recalling the definition of Lagrangian solution, we recall a further important concept.

\begin{definition}
Let  us fix $T>0$, $\bar{y}\in \Y$, and $(\overline{B}(t))_{t\in [0,T]}$ be a $\R^{d}$-valued Brownian motion. The transition map $\overline{F}_{\Lambda}(t,s,\bar{y})$ associated with the equation \eqref{eq:central} above, where  the initial condition is replaced by $(X(s),\lambda(s))=\overline{y}$ is defined through
\begin{align}\label{eq:flow}
\overline{F}_{\Lambda}(t,s,\overline{y})\coloneqq (X(t),\lambda(t)),
\end{align}
where $(X(t),\lambda(t))_{t\in [0,T]}$ is the unique solution to equation \eqref{eq:central} driven by $(\overline{B}(t))_{t\in [0,T]}$.
\end{definition}
We now recall what a Lagrangian solution is for an equation of the form \eqref{e:parabolic}.

\begin{definition}
Suppose that $\Lambda\in C^{0}([0,T];(\P_{1}(\Y),\W_{1}))$ and let $\overline{\Lambda}$ be any given initial condition for \eqref{e:parabolic}. We say that $\Lambda$ is a Lagrangian solution to \eqref{e:parabolic} starting from $\overline{\Lambda}$ if and only it satisfies that
\begin{align*}
\Lambda_{t}=\overline{F}_{\Lambda}(t,0,\cdot)_{\#}\overline{\Lambda} \hskip 0,3cm \text{for every $0\leq t\leq T$, $\cP$-a.s.,}
\end{align*}
where $\overline{F}_{\Lambda}(t,s,y)$ are the transition maps as defined in \eqref{eq:flow}.
\end{definition}
\begin{remark}
Observe that the measure $\Lambda_{t}$ can be seen as the action of the evaluation map $\ev_{t}$ which can be defined as follows. We set for each $t\in [0,T]$ the  evolution map $\ev_{t}:C^{0}([0,T];(\P_{1}(\Y),\W_{1}))\rightarrow \P_{1}(\Y)$ as
\begin{align}\label{evoper}
    \ev_{t}(\Lambda)\coloneqq \Lambda_{t}, \hskip 0,1cm \text{for all $\Lambda\in C^{0}([0,T];(\P_{1}(\Y),\W_{1}))$.}
\end{align}
In general, notice that the definition of the evolution map $\ev_{t}$ can be defined in $C^{0}([0,T];E)$, where $E$ is a generic Banach space, see \cite{AFMS}.
\end{remark}

\subsection{The multi-agent system: strong solutions for SDEs}\label{subsec5}
In this part, we introduce the model and we recall the notion of strong solutions for SDEs. Let us fix $T>0$, $\sigma>0$, and for $N\in\N$ we consider $N$ independent and identically distributed $\R^{d}$-valued Brownian motions $(B_{t}^{i})_{t\in [0,T]}$. In what follows, we consider a particle system of $N$ agents evolving according to

\begin{align}\label{system1}
\begin{dcases}
&X^{i}(t)=X_{0}^{i}+ \int_{0}^{t}v_{\Lambda_{s}^{N}}(X^{i}(s),\lambda^{i}(s))\de s+ \sqrt{2\sigma}B_{t}^{i},\\
&\lambda^{i}(t)=\lambda_{0}^{i}+\int_{0}^{t}\T_{\Lambda_{s}^{N}}(X^{i}(s),\lambda^{i}(s))\de s,
\end{dcases}
\end{align}
where $X^{i}(t)\in\R^{d}$, $\lambda^{i}(t)\in \P(U)$ for each $i=1,\ldots,N$, and 

\begin{align}\label{misura1}
\Lambda_{t}^{N}\coloneqq \frac{1}{N}\sum_{i=1}^{N}\delta_{(X^{i}(t),\lambda^{i}(t))}
\end{align}
is the empirical measure associated to the system. Since in this part 
 $N$ is fixed, we use $\Lambda_{t}$ rather than $\Lambda_{t}^{N}$ if confusion does not occur. We introduce the  vector-valued variable $\y\coloneqq (y^{1},\ldots,y^{N})\in \Y^{N}\subset \cE^{N}$, which we endow with the norm 
\begin{align*}
\norma{\y}_{\cE^{N}}\coloneqq \frac{1}{N}\sum_{i=1}^{N}\norma{y^{i}}_{\cE}.
\end{align*}
The natural space to study the well posedness of the system described through \eqref{system1} is $\cE^{N}$. However, notice that $\hat{\Y}^{N}\coloneqq (\R^{d})^{N}\times(\P(U))^{N}$, and $\hat{\cE}^{N}\coloneqq (\R^{d})^{N}\times(\F(U))^{N}$ can be  also endowed with the same norm $\norma{\y}_{\cE^{N}}$, making them more convenient spaces to analyze such a system. In what follows, we consider $\blambda\coloneqq (\lambda^{1},\ldots,\lambda^{N})\in (\P(U))^{N}$, and $\X\coloneqq (X^{1},\ldots, X^{N})\in (\R^{d})^{N}$. For each $\Psi\in \P_{1}(\Y)$, we consider the map $\v_{\Psi}^{N}: \hat{\Y}^{N}\rightarrow (\R^{d})^{N}$ which is defined through

\begin{align*}
\v_{\Psi}^{N}(\X,\blambda)\coloneqq (v_{\Psi}(X^{1},\lambda^{1}),\ldots,v_{\Psi}(X^{N},\lambda^{N})).
\end{align*}
Furthemore, we define the map $\bT_{\Psi}^{N}:\hat{\Y}^{N}\rightarrow (\F(U))^{N}$ as
\begin{align*}
\bT_{\Psi}^{N}(\X,\blambda)\coloneqq (\T_{\Psi}(X^{1},\lambda^{1}),\ldots,\T_{\Psi}(X^{N},\lambda^{N})),
\end{align*}
and we set $\bB(t)\coloneqq (B^{1}(t),\ldots, B^{N}(t))$ as the $d\times N$-valued Brownian motion. Then we write \eqref{system1} in the compact form

\begin{align}\label{system2}
\begin{dcases}
&\X(t)=\X_{0}+ \int_{0}^{t}\v_{\Lambda_{s}}^{N}(\X(s),\blambda(s))\de s+ \sqrt{2\sigma}\bB(t),\\
&\blambda(t)=\blambda_{0}+ \int_{0}^{t}\bT_{\Lambda_{s}}^{N}(\X(s),\blambda(s))\de s.
\end{dcases}
\end{align}
In the next, we suppose that ${\bf X}_{0}\in L^{2}(\Omega;(\R^{d})^{N})$, and ${\bf \lambda}_{0}\in (\P_{1}(U))^{N}$ for all $N\in\N$. Since the above is an SDE, we now recall what a strong solution is for \eqref{system2}.

\begin{definition}\label{defstrong}
Suppose that $(\Omega,\B,\cP)$ is a complete probability space endowed with the $\sigma$-algebra generated by $(\bB(t))_{t\in[0,T]}$ and that we denote as $({\mathcal F}_{t})_{t\in[0,T]}$. We say that an $\cE^{N}$-valued predictable process $\by(t)=(\X(t),\blambda(t))$, $t\in[0,T]$ is a strong solution of \eqref{system2} if $\by(t)$ satisfies $\cP$-a.s. \eqref{system2}, and it has a continuous version.   
\end{definition}

Let $(X,d_{X})$ be a metric space. We denote by $L^{2}(X)\coloneqq L^{2}(\Omega\times [0,T];X)$, and we endow it with the norm

\begin{align}\label{normavol1}
\norma{Y}_{L^{2}(X)}\coloneqq \E\left(\int_{0}^{T}\norma{Y(s)}_{X}^{2}\de s\right),\hskip 0,2cm Y\in L^{2}(X).
\end{align}
Further, we set $\M_{{\rm ad}}^{2}(X)\coloneqq L_{{\rm ad}}^{2}(\Omega\times [0,T];X)$ as the space of adapted processes with respect to the $\sigma$-algebra $({\mathcal F}_{t})_{t\in[0,T]}$ with values in $X$, and such that the norm \eqref{normavol1} is finite. 
\begin{remark}
In this work, we discuss the well-posedness of system \eqref{system2} for every choice of initial condition $(\X_{0},\blambda_{0})\in \hat{\Y}^{N}$. Furthermore, we are interested to looking forward the existence and uniqueness of solutions $\by\in \M_{{\rm ad}}^{2}(\hat{\Y}^{N})$.    
\end{remark}

\section{Main results}
\label{section3}
In this section, we collect our main results. In the next Theorem, we study the well-posedness of \eqref{system1}.
\begin{theorem}\label{importantprop}
Let us fix a filtered probability space $(\Omega,\B,\cP)$ endowed with a complete filtration $(\mathcal{F}_{t})_{t\in [0,T]}$ generated by the Brownian motion $(\bB(t))_{t\in[0,T]}$. Assume that for every $y\in \Y$ and $\Psi\in \P_{1}(\Y)$ the velocity field $v_{\Psi}:\Y\rightarrow \R^{d}$ satisfies (A\ref{A1})-(A\ref{A3}), and the operator $\T_{\Psi}$ satisfies (B\ref{B1})-(B\ref{B4}). Then, for every choice $\overline{\by}_{0}\coloneqq(\overline{\X}_{0},\overline{\blambda}_{0})$ as initial condition of \eqref{system2} such that $\overline{\by}_{0}\in L^{2}(\hat{\Y}^{N})$, the system \eqref{system2} has a unique solution $\by$ that belongs to $\M_{{\rm ad}}^{2}(\hat{\Y}^{N})$.
\end{theorem}
In the next theorem, we prove that Eulerian solutions are also Lagrangian solutions.
\begin{theorem}\label{thm:aux}
Let $r>0$ such that condition (C$1$) holds true. Furthermore, let us suppose that $\Lambda\in C^{0}([0,T];(\P_{1}(\Y),\W_{1}))$ and let $\overline{\Lambda}$ be any given initial condition for \eqref{e:parabolic} with $\overline{\Lambda}\in \P(B_{r}^{\Y})$. Moreover, suppose that $\Lambda$ is an Eulerian solution to \eqref{e:parabolic}. Then $\Lambda\in C^{0}([0,T];(\P(B_{r}^{\Y}),\W_{1}))$, and it is a Lagrangian solution for \eqref{e:parabolic}.
\end{theorem}
Our next result states the existence of a Eulerian solution for equation in \eqref{e:parabolic}, and it can be related to the mean-field limit of system \eqref{system1}.

\begin{theorem}\label{thm:mfg}
Let $r>0$ such that condition (C$1$) holds true, and let $\overline{\Lambda}\in \P(B_{r}^{\Y})$ be a fixed initial condition. The following hold true:

\begin{enumerate}[({i}1)]
\item there exists a unique Eulerian solution $\{\Lambda_{t}:t\in [0,T]\}$ in the sense of Definition \ref{dfn:Eulero} to \eqref{e:parabolic} starting from  $\overline{\Lambda}$.
\item Suppose that $\overline{\Lambda}^{N}\coloneqq \frac{1}{N}\sum_{i=1}^{N}\delta_{\overline{y}^{N,i}}$ is a sequence belonging to $\P(B_{r}^{\Y})$ such that
\begin{align*}
\lim_{N\rightarrow +\infty}\W_{1}(\overline{\Lambda}^{N},\overline{\Lambda})=0.
\end{align*}
Then for all $t\in [0,T]$
\begin{align*}
\lim_{N\rightarrow +\infty}\W_{1}(\Lambda_{t}^{N},\Lambda_{t})=0 \hskip 0,2cm \cP-a.s.,
\end{align*}
and where $\Lambda_{t}^{N}$ are the empirical measures given by \eqref{misura1} related to the unique solution of the system given by \eqref{system2} with initial datum $\{\overline{y}^{N,i}\}$.
\end{enumerate}
 \end{theorem}
We start by proving Theorem \ref{importantprop}. Let us remark that we cannot guarantee a priori that $\Lambda_{t}$ for $t>0$ is composed of independent and identically distributed random variables due to assumption (C$1$). Thus, approximating the system \eqref{system2} using the propagation of chaos, as proposed by Sznitman (see \cite{sznitman1991topics}), might not be possible. Indeed, one needs a system of couples $(\overline{X}^{1}(t), \overline{\lambda}^{1}(t)),\cdots,(\overline{X}^{N}(t), \overline{\lambda}^{N}(t))$ of independent and identically distributed random variables sufficiently close to \eqref{system2} in the sense of \cite[Definition 4.1]{Chaintron_2022a}. This, in fact, requires further knowledge on the structure of $\v$, and $\bT$ satisfying (A\ref{A1})-(A\ref{A3}), and (B\ref{B1})-(B\ref{B4}), respectively. For a self-contained review of the propagation of chaos approach, see also \cite{Chaintron_2022b}. Hence, our proof relies on Corollary \ref{Cor:MS}, which helps us overcome the local properties of $\v$, and $\bT$. Nevertheless, we need to construct probability measures $\blambda_{n}(t)$, $n\in\N$, which are necessary to obtain the solution of \eqref{system2} by means of the Picard iteration method. This convergence issue, in the case of an equation like \eqref{system2} with additional infinite-dimensional white noise in the second equation, requires a deeper analysis, which we postpone to future work.

\begin{proof}[Proof of Theorem \ref{importantprop}]
We prove this result by making use of a Picard iteration, and Corollary \ref{Cor:MS}. For each $n\in\N$, let us consider for any $\theta>0$ but fixed, we set $g_{\Lambda_{t}}(\X_{n}(t),\blambda_{n}(t))\coloneqq \blambda_{n}(t)+ \theta \bT_{\Lambda_{t}}^{N}(\X_{n}(t),\blambda_{n}(t))$ where
\begin{align}\label{Picard1}
\begin{dcases}
&\X_{n+1}(t)=\X_{0}+ \int_{0}^{t}\v_{\Lambda_{s}}^{N}(X_{n}(s), \blambda_{n}(s))\de s+ \sqrt{2\sigma}\bB(t),\\
&\blambda_{n+1}(t)=e^{-\frac{t}{\theta}} \blambda_{0}+ \frac{1}{\theta}\int_{0}^{t}e^{\frac{s-t}{\theta}}g_{\Lambda_{s}}(\X_{n}(s),\blambda_{n}(s))\de s,\\
&\X_{0}\coloneqq \overline{\X}_{0},\\
&\blambda_{0}\coloneqq \overline{\blambda}_{0}.
\end{dcases}
\end{align}

Notice that for each $n\in\N$, $\blambda_{n}\in (\P(U))^{N}$. Indeed,  since 

\begin{align*}
g_{\Lambda_{t}}(\X_{n}(t),\blambda_{n}(t)) \in (\P(U))^N, \quad \forall t \in [0,T],
\end{align*}
then by the convexity of $(\P(U))^N$, it holds that for all $t\in [0,T]$ 
\begin{align*}
\frac{\displaystyle\frac{1}{\theta}\int_0^t e^{\frac{s-t}{\theta}} g_{\Lambda_{s}}(\X_{n}(s),\blambda_{n}(s))\,\de s}{\displaystyle\frac{1}{\theta}\int_0^t e^{\frac{s-t}{\theta}}\,\de s} \in (\P(U))^N.
\end{align*}
Since $\frac{1}{\theta}\int_0^t e^{\frac{s-t}{\theta}}\,\de s = 1 - e^{-\frac{t}{\theta}}$, we get that
\begin{equation*}
\frac{1}{\theta}\int_0^t e^{\frac{s-t}{\theta}}g_{\Lambda_{s}}(\X_{n}(s),\blambda_{n}(s))\,\de s \in (1 - e^{-\frac{t}{\theta}})(\P(U))^N,
\end{equation*}
and then 
\begin{equation*}
\blambda_{n+1}(t) = e^{-\frac{t}{\theta}}\Lambda_0 + \frac{1}{\theta}\int_0^t e^{\frac{s-t}{\theta}}g_{\Lambda_{s}}(\X_{n}(s),\blambda_{n}(s))\,\de s \in (\P(U))^N.
\end{equation*}
We now use Corollary \ref{Cor:MS} to show that \eqref{Picard1} admits a solution. That is, we show that the differential equation

\begin{align}\label{diff:Picard1}
\begin{dcases}
&\de\X_{n+1}(t)=\v_{\Lambda_{t}}^{N}(\X_{n}(t),\blambda_{n}(t))+ \sqrt{2\sigma}\de\bB(t),\\
&\de\blambda_{n+1}(t)=-\frac{1}{\theta}\blambda_{n+1}(t) + \frac{1}{\theta}g_{\Lambda_{t}^{N}}(\X_{n}(t),\blambda_{n}(t)),\\
&\X_{0}\coloneqq \overline{\X}_{0},\\
&\blambda_{0}\coloneqq \overline{\blambda}_{0}.
\end{dcases}
\end{align}
admits a unique strong solution. Indeed, we can write such a system as

\begin{align*}
\de \by(t)=\hat{b}_{\Lambda_{t}}^{n}(\by(t)), \hskip 0,1cm \by=(\by_{1},\by_{2})  
\end{align*}
where
\begin{align*}
\hat{b}_{\Lambda_{t}}^{n}(\by(t))\coloneqq
\begin{pmatrix}
&\v_{\Lambda_{t}}^{N}(\X_{n}(t),\blambda_{n}(t))+ \sqrt{2\sigma}\de\bB(t)\\
&-\frac{1}{\theta}\by_{2}(t) + \frac{1}{\theta}g_{\Lambda_{t}}(\X_{n}(t),\blambda_{n}(t))
\end{pmatrix}.    
\end{align*}
Notice that $\hat{b}_{\Lambda_{t}}^{n}(c)$ is time-continuous for all $c\in \hat{\Y}^{N}$. On the other hand, for every $R>0$, we need to verify that there exists $\theta_{1}>0$ such that for $c\in\hat{\Y}^{N}$, $\norma{c}_{\hat{\Y}^{N}}\leq R$ then $c+\theta_{1}\hat{b}_{\Lambda_{t}}^{n}(c(t))\in \hat{\Y}^{N}$. Indeed, we may choose $\theta_{1}=\theta$, and thus 

\begin{align*}
c_{2} + \theta_{1}(-\frac{1}{\theta}c_{2} + \frac{1}{\theta}g_{\Lambda_{t}}(\X_{n}(t),\blambda_{n}(t)))=  g_{\Lambda_{t}}(\X_{n}(t),\blambda_{n}(t))\in (\P(U))^{N}.  
\end{align*}
Let us now notice that condition $(i')$ and $(ii')$ in Corollary \ref{Cor:MS} are satisfied. Indeed, by definition of $\hat{b}$ we have that

\begin{align*}
\norma{\hat{b}_{\Lambda_{t}}^{n}(\by_{1})-\hat{b}_{\Lambda_{t}}^{n}(\by_{2})}_{\hat{\Y}^{N}}\leq \frac{1}{\theta}\norma{\by_{1}-\by_{2}}_{\hat{\Y}^{N}}.
\end{align*}
In the same way, we can verify that 

\begin{align*}
\norma{\hat{b}_{\Lambda_{t}}^{n}(\by_{1})}_{\hat{\Y}^{N}}\leq \frac{1}{\theta}(1+\norma{\by_{1}}_{\hat{\Y}^{N}}).
\end{align*}
Therefore, for each fixed $n\in\N$, we have that there exists a unique curve $t\mapsto (\X_{n+1}(t),\blambda_{n+1}(t))$ solving the differential equation \eqref{diff:Picard1}. We now show that it is possible to construct the solution of \eqref{system2} as the limit of $(\X_{n+1}(t),\blambda_{n+1}(t))$, as $n\rightarrow +\infty$. First, let us now prove that $(\X_{1},\blambda_{1})\eqqcolon \by_{1}\in \M_{{\rm ad}}^{2}(\hat{\Y}^{N})$, and it has a continuous version. Let $t\in[0,T]$,  and $s,t\in [0,T]$ with $s<t$. By assumption (A\ref{A3}), one gets
\begin{align*}
\vert \X_{1}(t)-\X_{1}(s)\vert &\leq \int_{s}^{t}\vert \v_{\Lambda_{r}}^{N}(\X_{0}(r),\blambda_{0}(r))\vert\de r + \sqrt{2\sigma}\vert\bB(t)-\bB(s)\vert\\
&=\frac{1}{N}\sum_{i=1}^{N}\int_{s}^{t}\vert v_{\Lambda_{r}}(\X_{0}^{i}(r),{\bf \lambda}_{0}^{i}(r))\vert\de r+  \sqrt{2\sigma}\vert \bB(t)-\bB(s)\vert\\
&\leq M_{v}\int_{s}^{t}(1+\norma{\bar{Y}_{0}}_{\hat{\cE}^{N}}+ m_{1}(\Lambda_{r}^{N}))\de r+ \sqrt{2\sigma}\vert \bB(t)-\bB(s)\vert\\
&\leq M_{v}(2+\norma{\bar{Y}_{0}}_{\hat{\cE}^{N}})\vert t-s\vert+ \sqrt{2\sigma}\vert \bB(t)-\bB(s)\vert.
\end{align*}
Hence, by taking the square and expectation with respect to $\cP$, we get

\begin{align*}
\E\vert \X_{1}(t)-\X_{1}(s)\vert^{2}\leq 2(M_{v}(2+\norma{\bar{Y}_{0}}_{\hat{\cE}^{N}}))^{2}\vert t-s\vert^{2}+ 4\sigma \vert t-s\vert.
\end{align*}
Therefore, By  the Kolmogorov continuity theorem, $(\X_{1}(t))_{t\in [0,T]}$ has a continuous version. On the other hand, notice that

\begin{align*}
\vert \X_{1}(t)\vert &\leq \vert\overline{\X}_{0}\vert+ \int_{0}^{t}\vert\v_{\Lambda_{s}}^{N}(\X_{0}(s),\blambda_{0}(s))\vert\de s + \sqrt{2\sigma}\vert \bB(t)\vert\\
&=\vert\overline{\X}_{0}\vert+\frac{1}{N}\sum_{i=1}^{N}\int_{0}^{t}\vert v_{\Lambda_{s}}(\X^{i}(s),\blambda^{i}(s))\vert\de s+  \sqrt{2\sigma}\vert \bB(t)\vert\\
&\leq \vert\overline{\X}_{0}\vert+ M_{v}\int_{0}^{t}(1+\norma{\overline{Y}_{0}}_{\hat{\cE}^{N}}+ m_{1}(\Lambda_{s}^{N}))\de s+ \sqrt{2\sigma}\vert \bB(t)\vert.
\end{align*}
From here, we conclude that there exists a further positive constant still denoted as $C$ such that
\begin{align*}
\E\int_{0}^{T}\vert \X_{1}(t)\vert^{2}\de t\leq CT.
\end{align*}
Let us now proceed with $\blambda_{1}(t)$. Notice that for all $s,t\in[0,T]$ with $s\leq t$

\begin{align*}
\norma{\blambda_{1}(t)-\blambda_{1}(s)}_{(\F(U))^{N}}&\leq \vert e^{\frac{-t}{\theta}}- e^{\frac{-s}{\theta}}\vert \norma{\blambda_{0}}_{(\F(U))^{N}}\\
&\phantom{formula}+\tilde{M}_{\T}\int_{s}^{t}\hskip-0,2cm e^{\frac{r-t}{\theta}}(1+\vert \X_{0}(r)\vert+ m_{1}(\Lambda_{r}^{N}))\de r,
\end{align*}
where $\tilde{M}_{\T}$ is a positive constant depending on $M_{\T}$, and $\theta$. By taking the square and expectation with respect to $\cP$, one gets

\begin{align*}
\E\norma{\blambda_{1}(t)-\blambda_{1}(s)}_{(\F(U))^{N}}^{2}&\leq 2 \vert e^{\frac{-t}{\theta}}- e^{\frac{-s}{\theta}}\vert^{2} \norma{\blambda_{0}}_{(\F(U))^{N}}^{2}\\
&\phantom{formula}+2\tilde{M}_{\T}^{2}\E\left(\int_{s}^{t}(2+\vert \X_{0}(r)\vert)\de r\right)^{2}\\
&\leq 2 \vert e^{\frac{-t}{\theta}}- e^{\frac{-s}{\theta}}\vert^{2} \norma{\blambda_{0}}_{(\F(U))^{N}}^{2}+\\
&\phantom{formula}+4\tilde{M}_{\T}^{2}\left(4\vert t-s\vert^{2}+\vert t-s\vert \E\int_{s}^{t}\vert \X_{0}(r)\vert^{2}\de r\right),\\
\end{align*}
and thus by the Kolmogorov continuity theorem, $(\blambda_{1}(t))_{t\in [0,T]}$ has a continuous version. Furthermore, we can find a positive constant $C\coloneqq C(\theta,M_{\T}, {\bf \lambda}_{0},\X_{0},T)$ such that
\begin{align*}
\E\norma{\int_{0}^{T}\blambda_{1}^{2}(t)\de t}_{(\F(U))^{N}}\leq CT.
\end{align*}
Therefore $(\by_{1}(t))_{t\in [0,T]}\in \M_{{\rm ad}}^{2}(\hat{\Y}^{N})$. Let us now proceed by induction. That is, we suppose that  $(\X_{n},\blambda_{n})\eqqcolon \by_{n}\in  \M_{{\rm ad}}^{2}(\hat{\Y}^{N})$, and we prove that $ \by_{n+1}\in\M_{{\rm ad}}^{2}(\hat{\Y}^{N})$. As before, let us take $s,t\in [0,T]$ and $s<t$. Then

\begin{align*}
\vert \X_{n+1}(t)-\X_{n+1}(s)\vert &\leq \int_{s}^{t}\vert \v_{\Lambda_{s}}^{N}(\X_{n}(s),\blambda_{n}(s))\vert\de s + \sqrt{2\sigma}\vert \bB(t)-\bB(s)\vert\\
&=\frac{1}{N}\sum_{i=1}^{N}\int_{s}^{t}\vert v_{\Lambda_{s}}(\X_{n}^{i}(s),\blambda_{n}^{i}(s))\vert\de s+  \sqrt{2\sigma}\vert \bB(t)-\bB(s)\vert\\
&\leq M_{v}\int_{s}^{t}(1+\norma{\by_{n}}_{\hat{\cE}^{N}}+ m_{1}(\Lambda_{s}^{N}))\de s+ \sqrt{2\sigma}\vert \bB(t)-\bB(s)\vert.\\
\end{align*}
Hence, by taking the square and Jensen inequality, we get

\begin{align*}
\vert \X_{n+1}(t)-\X_{n+1}(s)\vert^{2}&\leq 2M_{v}^{2}\vert t-s\vert\int_{s}^{t}(1+\norma{\by_{n}}_{\hat{\cE}^{N}}+ m_{1}(\Lambda_{s}^{N}))^{2}\de s +4\sigma \vert \bB(t)-\bB(s)\vert^{2}\\
&\leq 4M_{v}^{2}\vert t-s\vert\int_{s}^{t}(4+\norma{\by_{n}}_{\hat{\cE}^{N}}^{2})\de s+4\sigma \vert \bB(t)-\bB(s)\vert^{2},
\end{align*}
and thus from this inequality and the inductive hypothesis, and the Kolmogorov continuity theorem, $(\X_{n+1}(t))_{t\in [0,T]}$ has a continuous version. On the other hand, notice that
\begin{align*}
\E\int_{0}^{T}\vert \X_{n+1}(t)\vert^{2}\de t\leq CT,
\end{align*}
for some positive constant $C$. Further, the same conclusion can be obtained for $\lambda_{n+1}$. Therefore $\by_{n}\in  \M_{{\rm ad}}^{2}(\hat{\Y}^{N})$ for all $n\in\N$. To conclude, We now need to prove that $(\by_{n})_{n}$ is a Cauchy sequence. Let us prove that there exists a positive constant $M>0$ such that

\begin{align}\label{f:dis1}
\E\left( \sup_{s\in [0,t]} \norma{\by_{n+1}(s)-\by_{n}(s)}_{\hat{\Y}^{N}}^{2}\right)\leq \frac{(Mt)^{n+1}}{(n+1)!}
\end{align}
for all $n\in\N$, and all $t\in [0,T]$.  By assumption (A\ref{A1}), let us take $R$ large enough such that $\Lambda_{t}^{N}\in \P(B_{R}^{\Y^{N}})$. We proceed by induction. Let us underline that the involved constants may varying from line to line. Notice that by assumption (A\ref{A3}), we can find a positive constant $C$ depending on $\norma{\by_{0}}_{\hat{\cE}^{N}}$ such that

\begin{align*}
\vert \X_{1}(t)-\X_{0}(t)\vert \leq Ct + \sqrt{2\sigma}\vert \bB(t)\vert,
\end{align*}
and thus by Doob's inequality one gets
\begin{align*}
\E \sup_{s\in[0,t]}\vert \X_{1}(s)-\X_{0}(s)\vert^{2}\leq Ct^{2},
\end{align*}
for some positive constant $C$. On the other hand, by assumption (B\ref{B2}) we have that

\begin{align*}
\vert \blambda_{1}(t)-\blambda_{0}\vert&\leq \vert \blambda_{0}\vert(1-\exp(-\frac{t}{\theta}))+ \tilde{M}_{\T}\int_{0}^{t}\exp\left(\frac{s-t}{\theta}\right)(2+\vert \X_{0}\vert)\de s\\
&\leq \frac{t}{\theta}\vert \blambda_{0}\vert+ \tilde{M}_{\T}(2+\vert \X_{0}\vert)t,
\end{align*}
and thus 
\begin{align*}
\E \sup_{s\in [0,t]}\vert \blambda_{1}(s)-\blambda_{0}\vert^{2}\leq Ct^{2}
\end{align*}
where $C$ is a positive constant depending on $M_{\T},\sigma,\theta, \E\norma{\by_{0}}^{2}$. Therefore

\begin{align*}
\E\sup_{s\in [0,t]}\norma{\by_{1}(s)-\by_{0}(s)}_{\hat{\Y}^{N}}^{2}\leq Ct^{2}
\end{align*}
where $C$ is a positive constant depending on $M_{\T},\sigma,\theta,T,\E\blambda_{0}^{2}$. We now suppose that \eqref{f:dis1} holds true for $n-1\in\N$, and then we prove it for $n\in\N$. Since assumptions (A\ref{A1})-(A\ref{A3}), and (B\ref{B1})-(B\ref{B4}) hold true, By Proposition \ref{morasolo1} item $(i)$, there exists a positive constant $L_{R}$ such that

\begin{align*}
&\vert \X_{n+1}(t)-\X_{n}(t)\vert\leq L_{R}\int_{0}^{t}\norma{\by_{n}(s)-\by_{n-1}(s)}_{\hat{\Y}^{N}}\de s,\\
&\vert\blambda_{n+1}(t)-\blambda_{n}(t)\vert\leq L_{R} \int_{0}^{t}\norma{\by_{n}(s)-\by_{n-1}(s)}_{\hat{\Y}^{N}}\de s,
\end{align*}
and thus
\begin{align*}
\norma{\by_{n+1}(t)-\by_{n}(t)}_{\hat{\Y}^{N}}\leq L_{R} \int_{0}^{t}\norma{\by_{n}(s)-\by_{n-1}(s)}_{\hat{\Y}^{N}}\de s.
\end{align*}
From this inequality, we get that for each $s\in [0,t]$

\begin{align*}
\norma{\by_{n+1}(s)-\by_{n}(s)}_{\hat{\Y}^{N}}\leq L_{R} \int_{0}^{s}\norma{\by_{n}(r)-\by_{n-1}(r)}_{\hat{\Y}^{N}}\de r.
\end{align*}
Hence, by Jensen inequality one gets 
\begin{align*}
\norma{\by_{n+1}(s)-\by_{n}(s)}_{\hat{\Y}^{N}}^{2}&\leq sL_{R}^{2}\int_{0}^{s}\norma{\by_{n}(r)-\by_{n-1}(r)}_{\hat{\Y}^{N}}^{2}\de r.
\end{align*}
and by taking the expected value, we obtain that for each $s\in [0,t]$

\begin{align*}
\E\norma{\by_{n+1}(s)-\by_{n}(s)}_{\hat{\Y}^{N}}^{2}&\leq tL_{R}^{2}\E\int_{0}^{t}\norma{\by_{n}(s)-\by_{n-1}(s)}_{\hat{\Y}^{N}}^{2}\de s\\
&\leq (tL_{R})^{2}\int_{0}^{t}\frac{(Ms)^{n}}{n!}\de s
\end{align*}
where in the last inequality we have used the inductive hypothesis. Therefore, we conclude that

\begin{align*}
\E\sup_{s\in[0,t]}\norma{\by_{n+1}(s)-\by_{n}(s)}_{\hat{\Y}^{N}}^{2}\leq (TL_{R})^{2}\frac{(Mt)^{n+1}}{(n+1)!} 
\end{align*}
By redefining the constant $M$, we obtain the desired inequality \eqref{f:dis1}. A direct consequence of \eqref{f:dis1} is that by Markov inequality
\begin{align*}
\cP\left(\sup_{t\in[0,T]}\norma{\by_{n+1}(t)-\by_{n}(t)}_{\hat{\Y}^{N}}^{2}>2^{-n}\right)&\leq 4^{n}\E\left(\sup_{t\in[0,T]}\norma{\by_{n}(t)-\by_{n-1}(t)}_{\hat{\Y}^{N}}^{2}\right)\\
&\leq 4^{n}\frac{(MT)^{n+1}}{(n+1)!}
\end{align*}
Since the series $\sum_{n\in\N}4^{n}\frac{(MT)^{n+1}}{(n+1)!}$ converges, then by Borel-Cantelli

\begin{align*}
\cP\left(\limsup_{n\in\N}\{\omega\in \Omega: \sup_{t\in[0,T]}\norma{\by_{n+1}(t)-\by_{n}(t)}_{\hat{\Y}^{N}}^{2}>2^{-n}\}\right)=0.
\end{align*}
Then, we may define $\by\coloneqq \lim_{n\rightarrow +\infty}\by_{n}\in \M_{{\rm ad}}^{2}(\hat{\Y}^{N})$. Notice that this process solves 

\begin{align*}
\begin{dcases}
&\X(t)=\X_{0}+ \int_{0}^{t}\v_{\Lambda_{s}}(\X(s),\blambda(s))\de s+ \sqrt{2\sigma}\bB(t),\\
&\blambda(t)=e^{-\frac{t}{\theta}}\blambda_{0}+ \frac{1}{\theta}\int_{0}^{t}e^{\frac{s-t}{\theta}}g_{\Lambda_{s}}(\X(s),\blambda(s))\de s,\\
&\X_{0}\coloneqq \overline{\X}_{0},\\
&\blambda_{0}\coloneqq \overline{\blambda}_{0}.
\end{dcases}
\end{align*}
We now need to come back to the original problem. Indeed, we now aim to show that $\by$ solves \eqref{system2}. The only thing that we need to prove is that 

\begin{align*}
&\blambda(t)=e^{-\frac{t}{\theta}}\blambda_{0}+ \frac{1}{\theta}\int_{0}^{t}e^{\frac{s-t}{\theta}}g_{\Lambda_{s}}(\X(s),\blambda(s))\de s,
&\overline{\blambda}(t)=\blambda_{0}+ \int_{0}^{t}\bT_{\Lambda_{s}}^{N}(\X(s),\blambda(s))\de s,
\end{align*}
coincide. Let us define for each $t\in [0,T]$ the function $\phi(t)\coloneqq\blambda(t)-\overline{\blambda}(t)$. Notice that

\begin{align*}
\frac{\de}{\de t}\blambda(t)=-\frac{1}{\theta}\blambda(t)+\frac{1}{\theta}\left(\blambda(t)+\theta \bT_{\Lambda_{t}}(\X(t),\blambda(t)) \right)= \bT_{\Lambda_{t}}(\X(t),\blambda(t)),
\end{align*}
and thus by  noting that $\frac{\de}{\de t}\overline{\blambda}(t)=\bT_{\Lambda_{t}}(\X(t),\blambda(t))$ then $\frac{\de}{\de t}\phi(t)=0$ for all $t\in [0,T]$. Thus $\phi(t)=c$ for all $t\in [0,T]$. So that, by the initial condition, we conclude that $c=0$, and we are done.
\end{proof}
\subsection{Mean-field limit}
We now analyze the asymptotic behavior of \eqref{system2} as $N\rightarrow +\infty$. Specifically, we want to demonstrate the existence of an Eulerian solution as $N\rightarrow +\infty$ as stated in Theorem \ref{thm:mfg}. To this aim, we first derive the Theorem \ref{thm:aux}, and some preliminary results.

\begin{proof}[Proof of Theorem \ref{thm:aux}]
 In order to prove this result, let us notice that $\mathcal L_{\bar{\Lambda}}$ defines a strongly continuous semigroup. Indeed, by condition (C$1$) one has that 
 \begin{align*}
 C_{0}(\Y)=\overline{{\rm range}(\eta I-\mathcal L_{\bar{\Lambda}})}={\rm range}(\eta I-\overline{\mathcal L}_{\bar{\Lambda}}).
 \end{align*}
 Then by the Lummer-Phillips theorem, we can conclude that $\overline{\mathcal L}_{\bar{\Lambda}}$ generates a strongly continuous semigroup of contractions on $C_{0}(\Y)$ that we denoted as $\{S_{t}\}_{t\geq 0}$. By Theorem \ref{wellpod}, $S_{t}^{\ast}\bar{\Lambda}$ is well-defined and defines a semigroup of linear and continuous operators on $(C_{b}(\Y))^{\ast}$. Moreover, since $\bar{\Lambda}\in \P(B_{r}^{\Y})$ and $S_{t}$ is a contraction, then $S_{t}\bar{\Lambda}\in \P(B_{r}^{\Y})$. On the other hand, since \eqref{newdiffeq} admits a unique solution, and $\Lambda$ satisfies \eqref{e:parabolic} then $\Lambda_{t}=S_{t}^{\ast}\bar{\Lambda}$. Moreover, by remark \ref{itoderivation}, we conclude that $S_{t}^{\ast}\bar{\Lambda}=(X(t),\lambda(t))_{\#}\bar{\Lambda}$, and we are done.
\end{proof}
We  now aim to prove Theorem \ref{thm:mfg}. We start by showing that the sequences $\{\Lambda_{t}^{N}\}_{N\in\N}$ $\{X^{N}(t)\}_{N\in\N}$, and $\{\lambda^{N}(t)\}_{N}$ are tight. To this aim, we use the so-called Aldous criteria, see \cite[Theorem 16.10]{billi1999}, see also Appendix \ref{Aldoscrit1} below.

\begin{proposition}\label{prop:tigh1}
Let us consider the same setting as in Theorem \ref{importantprop}. and suppose that  for each $i=1,\ldots,N$, $X_{0}^{i}$ is distributed according to some fixed probability measure $\overline{\mu}$. Then the following holds true.

\begin{enumerate}
    \item There exists a positive constant $C$ independent of $N$ such that

\begin{align}\label{moments:2}
 \sup_{i=1,\ldots,N}\left\{\sup_{t\in [0,T]}\left\{\E\vert X^{i}(t)\vert^{2}+ \E\norma{ \lambda^{i}(t)}^{2}\right\}\right\} <C;
\end{align}
\item For all $\varepsilon>0,\eta>0$ there exists $\delta_{0}>0$ and $N_{0}\in \N$ such that for all $N\geq N_{0}$, and for all discrete valued $\sigma\left(X^{N}(s): s\in [0,T] \right)$-stopping times $\beta$ with $0\leq \beta\leq \beta+\delta_{0}\leq T$, it is true that
\begin{align}\label{moments:3}
 \sup_{0\leq\delta\leq \delta_{0}}\cP\left(\vert X^{N}(\beta+\delta)-X^{N}(\beta)\vert \geq \eta\right)\leq \varepsilon;
\end{align}
and the same inequality holds true for the sequence $\{{\rm Law}(\lambda_{t}^{N})\}_{N\in \N}$.
\item For all $t\in [0,T]$, the sequence of probability distributions $\{{\rm Law}(X_{t}^{N})\}_{N\in \N}$ is tight as a sequence of probability measures on $\R^{d}$, and the same holds true for the sequence of probability distributions $\{{\rm Law}(\lambda_{t}^{N})\}_{N\in \N}$ as a sequence of probability measures on $\P(U)$.
\end{enumerate}
\end{proposition}
\begin{remark}\label{rmk:1}
Let us recall that $\M(\Y)$ is a locally convex Hausdorff topological space, and its topological dual is $C_{b}(\Y)$. Furthermore, $\P(\Y)$ is a closed subset of $\M(\Y)$. This fact is crucial when considering the asymptotic behavior of $\Lambda_{t}^{N}$, the empirical measure of $(X^{i}(t),\lambda^{i}(t))_{i=1}^{N}$, especially in the case of identically distributed random variables. Despite the dependencies introduced by the interaction terms $v_{\Lambda_{t}^{N}}$, and $\T_{\Lambda_{t}^{N}}$, we will show that $\Lambda_{t}^{N}$ converges to some $\overline{\Lambda}_{t}$ $\cP$-a.s.. A strategy to establish this convergence is to prove that the sequence $(\Lambda_{t}^{N})_{N\in \N}$ satisfies a large deviation principle (LDP). This is feasible because the closed convex hull of any compact set $K\subset \P(\Y)$ is compact with respect to the relative topology induced by $\M(\Y)$. To demonstrate this LDP, we can apply \cite[Theorem 6.1.3]{dembo93}, which establishes that the sequence $(\Lambda_{t}^{N})_{N}$ satisfies a weak large deviation principle in $\P(\Y)$, and the convex rate function is given by

\begin{align*}
\Phi_{t}^{\ast}(\nu)\coloneqq\sup_{\varphi \in C_{b}(\Y)}\left\{\nm{\varphi,\nu}-\Phi_{t}(\nu)\right\}, \quad \nu\in \P(\Y),
\end{align*}
where, for $\varphi\in C_{b}(\Y)$,
\begin{align*}
\Phi_{t}(\nu)\coloneqq \log\E\left(\exp\nm{\varphi,\delta_{y_{t}^{1}}}\right)=\log\E(\exp(\varphi(y_{t}^{1}))),
\end{align*}
at least in the case of independent, and identically distributed random variables. Subsequently, by \cite[Lemma 6.2.6]{dembo93}, we have that $(\Lambda_{t}^{N})_{N\in\N}$ is exponentially tight, and thus, a full large deviation principle follows for all $t\in [0,T]$. Therefore, $\Lambda_{t}^{N} \rightarrow \overline{\Lambda}_{t}$ $\cP$-a.s., for all $t\in [0,T].$
\end{remark}
Notice that Remark \ref{rmk:1} provides an interesting description of the convergence for the empirical measure $\Lambda_{t}^{N}$. However, it does not guarantee time regularity properties for the limit $\overline{\Lambda}_{t}$. Therefore, it motivates the study of \eqref{system1} by considering the Eulerian and Lagrangian solutions.

\begin{proof}[Proof of Proposition \ref{prop:tigh1}]
Let us fix $N\in\N$, and  for each $i=1,\ldots,N$ consider the system \eqref{system1}. We start by proving that \eqref{moments:2} holds true. Indeed, notice that

\begin{align*}
\vert X^{i}(t)\vert&\leq \vert X^{i}(0)\vert+ \int_{0}^{t} \vert v_{\lambda_{s}^{N}}(X^{i}(s),\lambda^{i}(s))\vert\de s+ \sqrt{2\sigma}\vert B_{t}^{i}\vert\\
&\leq \vert X^{i}(0)\vert+ M_{v}\int_{0}^{t}\left(1+\norma{(X^{i}(s),\lambda^{i}(s))} + m_{1}(\Lambda_{s}^{N})\right) \de s+ \sqrt{2\sigma}\vert B_{t}^{i}\vert.
\end{align*}
On the other hand, we have that

\begin{align*}
\norma{\lambda^{i}(t)}&\leq \norma{\lambda^{i}(0)}+ \int_{0}^{t} \vert \T_{\lambda_{s}^{N}}(X^{i}(s),\lambda^{i}(s))\vert\de s\\
&\leq \norma{\lambda^{i}(0)}+ M_{\T}\int_{0}^{t}\left(1+\norma{X^{i}(s)} + m_{1}(\Lambda_{s}^{N})\right) \de s\\
&\leq \norma{\lambda^{i}(0)}+ M_{\T}\int_{0}^{t}\left(1+\norma{(X^{i}(s),\lambda^{i}(s))} + m_{1}(\Lambda_{s}^{N})\right) \de s.
\end{align*}

Therefore, one obtains that

\begin{align*}
\frac{1}{N}\sum_{i=1}^{N}\norma{(X^{i}(t),\lambda^{i}(t))}_{\Y}&\leq \frac{1}{N}\sum_{i=1}^{N}\norma{(X^{i}(0),\lambda^{i}(0))}_{\Y}+ C_{v,\T}t+ \sqrt{2\sigma}\vert B_{t}^{i}\vert+\\
&+C_{v,\T}\int_{0}^{t}\left(\frac{1}{N}\sum_{i=1}^{N}\norma{(X^{i}(s),\lambda^{i}(s))}_{\Y} + m_{1}(\Lambda_{s}^{N})\right) \de s\\
&\leq \frac{1}{N}\sum_{i=1}^{N}\norma{(X^{i}(0),\lambda^{i}(0))}_{\Y}+ C_{v,\T}t+ \sqrt{2\sigma}\vert B_{t}^{i}\vert+\\
&\phantom{form}+C_{v,\T}\int_{0}^{t}\left(\frac{2}{N}\sum_{i=1}^{N}\norma{(X^{i}(s),\lambda^{i}(s))}_{\Y}\right) \de s,
\end{align*}
where $C_{v,\T}\coloneqq \max\{M_{v},M_{\T}\}$. By taking the expected value, we derive that

\begin{align*}
\frac{1}{N}\sum_{i=1}^{N}\E\norma{(X^{i}(t),\lambda^{i}(t))}_{\Y}&\leq \frac{1}{N}\sum_{i=1}^{N}\E\norma{(X^{i}(0),\lambda^{i}(0))}_{\Y}+ C_{v,\T}t+\\
&+C_{v,\T}\int_{0}^{t}\left(\frac{1}{N}\sum_{i=1}^{N}\E\norma{(X^{i}(s),\lambda^{i}(s))}_{\Y} + \E m_{1}(\Lambda_{s}^{N})\right) \de s.
\end{align*}

By applying Gr\hol{o}nwall inequality, it implies that

\begin{align*}
\frac{1}{N}\sum_{i=1}^{N}\E\norma{(X^{i}(t),\lambda^{i}(t))}_{\Y} &\leq \left( \frac{1}{N}\sum_{i=1}^{N}\E\norma{(X^{i}(0),\lambda^{i}(0))}_{\Y}\right)\times\exp(C_{v,\T}T)+\\
&\phantom{form}+\left(C_{v,\T}t+ \int_{0}^{t}\E m_{1}(\Lambda_{s}^{N})\de s\right)\times\exp(C_{v,\T}T),
\end{align*}
and thus

\begin{align*}
\E m_{1}(\Lambda_{t}^{N})&\leq \frac{1}{N}\sum_{i=1}^{N}\E\norma{(X^{i}(t),\lambda^{i}(t))}_{\Y}\\
&\leq \left( \frac{1}{N}\sum_{i=1}^{N}\E\norma{(X^{i}(0),\lambda^{i}(0))}_{\Y}\right)\times\exp(C_{v,\T}T)+\\
&\phantom{form}+\left(C_{v,\T}t+ \int_{0}^{t}\E m_{1}(\Lambda_{s}^{N})\de s\right)\times\exp(C_{v,\T}T).
\end{align*}
Therefore,  we obtain by Gr\hol{o}nwall inequality that

\begin{align}\label{condizioneimp}
\E m_{1}(\Lambda_{t}^{N}) \leq \exp(C_{v,\T})\exp(\exp(C_{v,\T})t)\left(C_{v,\T}t+  \frac{1}{N}\sum_{i=1}^{N}\E\norma{(X^{i}(0),\lambda^{i}(0))}_{\Y}\right)\eqqcolon C_{v,\T,t},
\end{align}
for all $t\in [0,T]$.
From here, we obtain  by Jensen inequality, and the It\hol{o} identity that

\begin{align*}
&\E\vert X^{i}(t)\vert^{2}\leq \E\vert X^{i}(0)\vert^{2}+ C_{v}t\int_{0}^{t}\left(1+\E \norma{(X^{i}(s),\lambda^{i}(s))}^{2}+ C_{v,\T,T}^{2}\right)\de s + 2\sigma t,\\
&\E\norma{\lambda^{i}(t)}^{2}\leq \E\norma{\lambda^{i}(0)}^{2}+ M_{\T}t\int_{0}^{t}\left(1+\norma{(X^{i}(s),\lambda^{i}(s))} + C_{v,\T,T}^{2}\right) \de s.
\end{align*}
Therefore, by Gr\hol{o}nwall inequality, we can find a positive constant $C\coloneqq C(v,\T,T)$ independent of $N$, and depending on $T,v,\T$ such that \eqref{moments:3} holds true. Let us proceed with the second item. Let $\beta$ be a generic stopping time, and let $\delta>0$. We have that for each $i=1,\ldots,N,$

\begin{align*}
  \vert X^{i}(\beta+\delta)-X^{i}(\beta)\vert &\leq \int_{\beta}^{\beta+\delta}v_{\Lambda_{s}^{N}}(X^{i}(s),\lambda^{i}(s))\de s+\sqrt{2\sigma}\vert B_{\delta}^{i}\vert \\
  &\leq M_{v}\int_{\beta}^{\beta+\delta}\left( 1+ \norma{(X^{i}(s),\lambda^{i}(s))}+ m_{1}(\Lambda_{s}^{N})\right)\de s +\sqrt{2\sigma}\vert B_{\delta}^{i}\vert .
\end{align*}
Then by taking the expectation in both sides of the previous inequality, we get

\begin{align*}
\E \vert X^{i}(\beta+\delta)-X^{i}(\beta)\vert \leq C\delta   
\end{align*}
where  we have used that $\E\norma{X^{i}(t)}<+\infty$, and $\E m_{1}(\Lambda_{t}^{N})<+\infty$ for all $t\in [0,T]$, and thus $C$ is a positive constant independent of $N$, and t. Hence, by letting

\begin{align*}
 U_{\varepsilon}\coloneqq \left\{\omega \in \Omega:  \vert X^{i}(\beta+\delta,\omega)-X^{i}(\beta,\omega) \vert \leq \frac{C}{\varepsilon}\right\},   
\end{align*}
we have that by Markov inequality that

\begin{align*}
\cP\left(U_{\varepsilon}^{c}> \frac{C}{\varepsilon}\right) \leq \frac{\varepsilon \E \vert X^{i}(\beta+\delta)-X^{i}(\beta)\vert}{C} \leq \varepsilon, 
\end{align*}
and this completes the proof of \eqref{moments:3}. The same reasoning can be applied for $\{{\rm Law}(\lambda_{t}^{N})\}_{N\in \N}$. Let us proceed with the third item. Notice that by Markov inequality, one has that
\begin{align*}
\cP\left(\sup_{t\in [0,T]}\norma{(X^{i}(t),\lambda^{i}(t))}\geq a\right) \leq \frac{\E\left(\sup_{t\in [0,T]}\norma{(X^{i}(t),\lambda^{i}(t))}\right)^{2}}{a^{2}},    
\end{align*}
for all $a>0$. Then by letting $a\rightarrow +\infty$, we have by the previous estimates that

\begin{align*}
\lim_{a\rightarrow +\infty}\cP\left(\sup_{t\in [0,T]}\norma{(X^{i}(t),\lambda^{i}(t))}\geq a\right)=0. 
\end{align*}
Therefore, we can conclude that the sequences $\{X^{i}(t): t\in [0,T]\}_{i\in \N}$, $\{{\rm Law}(\lambda_{t}^{N})\}_{N\in \N}$ are tight.
\end{proof}
\subsection{Stability}
In the next, we prove that starting from a probability measure $\bar{\Lambda}\in \P(B_{r}^{\Y})$ for some $r>0$, Lagrangian solutions to \eqref{e:parabolic} belongs to $\P_{1}(B_{r}^{\Y})$.

\begin{lemma}\label{model:lemma}
 Let us take $r>0$, and $\overline{\Lambda}\in \P_{c}(B_{r}^{Y})$. Furthermore, let us suppose that the velocity field $v_{\Psi}:\Y\rightarrow \R^{d}$ satisfies (A\ref{A1})-(A\ref{A3}), and the operator $\T_{\Psi}$ satisfies (B\ref{B1})-(B\ref{B4}). Suppose that $\Lambda \in C^{0}([0,T];(\P_{1}(Y),\W_{1}))$ is a Lagrangian solution for \eqref{e:parabolic} with initial condition $\overline{\Lambda}$. Then $\Lambda_{t}\in \P_{1}(B_{r}^{\Y})$ for all $t\in [0,T]$.
\end{lemma}
\begin{proof}
We want to prove that $\Lambda_{t}\in \P_{1}(B_{r}^{\Y})$. To do that, we prove that $m_{1}(\Lambda_{t})\leq m_{1}(\overline{\Lambda})$. Indeed, since by hypothesis  $\Lambda$ is a Lagrangian solution, and by using \eqref{evoper} $\Lambda_{t}=(\ev_{t})_{\#}\overline{\Lambda}$ the following inequalities hold true.  Let $X\coloneqq C([0,T];\Y)$ so that

\begin{align*}
 m_{1}(\Lambda_{t})&=\int_{\Y}\norma{y}_{\Y}\de\Lambda_{t}(y)\\
&=\int_{X}\norma{\ev_{t}(\varphi)}_{\Y}\de\overline{\Lambda}(\varphi)=\int_{X}\norma{\varphi(t)}_{\Y}\de\overline{\Lambda}(\varphi)\\
&\leq \int_{X}\sup_{t\in [0,T]}\norma{\varphi(t)}\de\overline{\Lambda}
(\varphi)\\
&\int_{X}\norma{\varphi}_{\infty}\de\overline{\Lambda}
(\varphi)=m_{1}(\overline{\Lambda}),
\end{align*}
and thus $m_{1}(\Lambda_{t})\leq m_{1}(\overline{\Lambda})\leq r$. Therefore, we obtain that $\Lambda_{t}\in \P_{1}(B_{r}^{\Y})$ for all $t\in [0,T]$.
\end{proof}
\begin{remark}
Notice that if we directly use the system \eqref{eq:limite}, by Gr\hol{o}nwall inequality we can prove that  $\Lambda_{t}\in \P_{1}(B_{R}^{\Y})$ where
\begin{align*}
m_{1}(\Lambda_{t})\leq (r + Mt)\exp(2Mt)\eqqcolon R.
\end{align*}
\end{remark}

\begin{lemma}\label{lemma:estab}
Suppose that $\Lambda^{1},\Lambda^{2}$ are two Lagrangian solutions for \eqref{e:parabolic} starting from $\overline{\Lambda}^{1}, \overline{\Lambda}^{2}\in \P_{1}(B_{r}^{\Y})$ for some fixed $r>0$. Then  

\begin{align*}
 \W_{1}(\Lambda_{t}^{1},\Lambda_{t}^{2})\leq C\exp(tL_{R})\W_{1}(\overline{\Lambda}^{1},\overline{\Lambda}^{2}),  
\end{align*}
where $C\coloneqq 1+ \exp(\exp(TL_{R}))\exp(TL_{R})$.
\end{lemma}
\begin{proof}
    Let us fix $r>0$, and consider two Lagrangian solutions $\Lambda^{1},\Lambda^{2}$ for \eqref{e:parabolic} starting from $\bar{\Lambda}^{1}, \bar{\Lambda}^{2}\in \P_{1}(B_{r}^{\Y})$. By using \eqref{eq:limite}, we get

\begin{align*}
\norma{y^{1}(t)-y^{2}(t)}_{\Y}&\leq \norma{y_{0}^{1}-y_{0}^{2}}_{\Y}+ \int_{0}^{t}\norma{b_{\Lambda_{s}^{1}}(y^{1}(s))-b_{\Lambda_{s}^{2}}(y^{2}(s))}_{\Y}\de s\\
&\leq \norma{y_{0}^{1}-y_{0}^{2}}_{\Y}+ \int_{0}^{t}\norma{b_{\Lambda_{s}^{1}}(y^{1}(s))-b_{\Lambda_{s}^{1}}(y^{2}(s))}_{\Y}\de s\\
&\phantom{formula}+\int_{0}^{t}\norma{b_{\Lambda_{s}^{1}}(y^{2}(s))-b_{\Lambda_{s}^{2}}(y^{2}(s))}_{\Y}\de s\\
&\leq \norma{y_{0}^{1}-y_{0}^{2}}_{\Y}+ L_{R}\int_{0}^{t}\norma{y^{1}(s)-y^{2}(s)}_{\Y}\de s + L_{R}\int_{0}^{t}\W_{1}(\Lambda_{s}^{1},\Lambda_{s}^{2})\de s.
\end{align*}
Let us set $\alpha(t)\coloneqq \norma{y_{0}^{1}-y_{0}^{2}}_{\Y}+ L_{R}\int_{0}^{t}\norma{y^{1}(s)-y^{2}(s)}_{\Y}\de s$. By Gr\hol{o}nwall inequality, we obtain 

\begin{align*}
\norma{y^{1}(t)-y^{2}(t)}_{\Y}&\leq \alpha(t)+L_{R}\int_{0}^{t}\alpha(s)\exp(L_{R}(t-s))\de s \\
&\leq \alpha(t)\left(1+ L_{R}\int_{0}^{t}\exp(L_{R}(t-s))\de s\right)\\
&=\alpha(t)\exp(tL_{R}).
\end{align*}
 From the previous inequality, we then get that
\begin{align*}
\W_{1}(\Lambda_{t}^{1},\Lambda_{t}^{2})\leq \exp(TL_{R})\left(\W_{1}(\Lambda_{0}^{1},\Lambda_{0}^{2})+L_{R}\int_{0}^{t}\W_{1}(\Lambda_{s}^{1},\Lambda_{s}^{2})\de s\right).
\end{align*}
 By applying Gr\hol{o}nwall inequality, we obtain

 \begin{align*}
\W_{1}(\Lambda_{t}^{1},\Lambda_{t}^{2})\leq &\exp(tL_{R})\W_{1}(\Lambda_{0}^{1},\Lambda_{0}^{2})\\
&+\W_{1}(\Lambda_{0}^{1},\Lambda_{0}^{2})\int_{0}^{t}L_{R}\exp(2L_{R}s)\exp\left(L_{R}\int_{s}^{t}\exp(L_{R}r)\de r\right)\de s\\
&\leq \exp(tL_{R})\W_{1}(\Lambda_{0}^{1},\Lambda_{0}^{2}) + \\
&+\exp(\exp(L_{R}T))\exp(TL_{R})\W_{1}\Lambda_{0}^{1},\Lambda_{0}^{2})L_{R}\int_{0}^{t}\exp(sL_{R})\de s\\
&\leq \exp(tL_{R})\W_{1}(\Lambda_{0}^{1},\Lambda_{0}^{2}) + \\
&+\exp(\exp(L_{R}T))\exp(TL_{R})\W_{1}(\Lambda_{0}^{1},\Lambda_{0}^{2})\exp(tL_{R}).
 \end{align*}
\end{proof}
We let for all $i=1,\ldots,N$,
\begin{align*}
b_{\Psi}(y(t))=
\begin{pmatrix}
&\hskip -0,3cm v_{\Psi}(y(t))\\
&\hskip -0,3cm \T_{\Psi}(y(t))
\end{pmatrix},
\hskip 0,2cm
W^{i}(t)\coloneqq 
\begin{pmatrix}
&\hskip -0,3cm -\sqrt{2\sigma}\frac{\de}{\de t}B^{i}(t)\\
&\hskip -0,3cm 0
\end{pmatrix},
\end{align*}
and we write equation \eqref{system1} as
\begin{align}\label{mod:system1}
\begin{aligned}
&\frac{\de}{\de t}y^{i}(t)= b_{\Lambda_{t}^{N}}(y^{i}(t))-W^{i}(t).
\end{aligned}
\end{align}
Let us notice that by \eqref{condizioneimp}, we have that our sequence $\{\Lambda_{t}^{N}\}_{t\in [0,T]}$ belongs to some ball $B_{R}^{\Y}$ for some positive radius $R$ independent of $N$. Since $U$ is assumed compact, we have that $B_{R}^{\Y}$ is compact, and then we have that $\Lambda_{t}$ converges weakly to some $\Lambda_{t}\in B_{R}^{\Y} $ as $N\rightarrow +\infty$. Then as $N\rightarrow +\infty$, we obtain the SDE

\begin{align}\label{eq:limite}
\begin{aligned}
&X(t)=\overline{X}_{0}+ \int_{0}^{t} v_{\Lambda_{t}}(X(s),\lambda(s))\de s + \sqrt{2\sigma}\overline{B}(t),\\
&\lambda(t)=\overline{\lambda}_{0}+ \int_{0}^{t}\T_{\Lambda_{t}}(X(s),\lambda(s))\de s,
\end{aligned}
\end{align}
where $X(0)=\overline{X}_{0}\in L^{2}(\Omega;\R^{d})$, and $\lambda(0)=\overline{\lambda}_{0}\in L^{2}(\Omega;\P(U))$, $\overline{B}$ is a suitable Brownian motions. We are now ready to prove Theorem \ref{thm:mfg}.
\begin{proof}[Proof of Theorem \ref{thm:mfg}]
 Let us take $r>0$ such that condition (C$1$) holds true, and let $\overline{\Lambda}\in \P(B_{r}^{\Y})$ be a fixed initial condition for \eqref{e:parabolic}.  In what follows, we start by proving item $(ii)$. Consider a sequence $\bar{\Lambda}^{N}$ of atomic measures such that 
 \begin{align*}
     \lim_{N\rightarrow+\infty}\W_{1}(\bar{\Lambda}^{N},\bar{\Lambda})=0.
 \end{align*}
 Notice that by letting $\bar{\Lambda}^{N}\coloneqq \frac{1}{N}\sum_{i=1}^{N}\delta_{y_{0}^{i}}$, where $y_{0}^{i}$ are the initial datum for
\begin{align*}
\dot{y}^{i}(t)= \hat{b}_{\Lambda^{i}}^{i}(y^{i}(t))-W(t),
\end{align*}
with
 \begin{align*}
     W(t)\coloneqq 
\begin{pmatrix}
&\hskip -0,3cm -\sqrt{2\sigma}\dot{B}^{i}(t)\\
&\hskip -0,3cm 0
\end{pmatrix},
 \end{align*}
by Theorem \ref{thm:aux}, we obtain that $\Lambda_{t}^{N}=(\overline{F}(t,0,\cdot))_{\#}\overline{\Lambda}^{N}$ is a Lagrangian and Eulerian solution to \eqref{e:parabolic} starting from $\overline{\Lambda}^{N}$. Then by Lemma \ref{lemma:estab}, we have that there exists a positive constant $C$ such that

 \begin{align}\label{benfatt}
     \W_{1}(\Lambda_{t}^{N},\Lambda_{t}^{M})\leq C\exp(tL_{R})\W_{1}(\bar{\Lambda}^{N},\overline{\Lambda}^{M})
 \end{align}
 for all $t\in [0,T]$, and all $N,M\in \N$. Moreover, by Lemma \ref{model:lemma}, we get that $\Lambda^{N}\in C([0,T];\P_{1}(B_{r}^{\Y},\W_{1}))$. Furthermore, since by assumption $\lim_{N\rightarrow +\infty}\W_{1}(\bar{\Lambda}^{N},\overline{\Lambda})=0$, and by Remark \ref{compacty} $\P(B_{r}^{\Y})$ is compact, by \eqref{benfatt} one obtains that the sequence $\{\Lambda_{t}^{N}\}_{N\in\N}$ $\W_{1}$-converges to some limit $\Lambda_{t}\in \P(B_{r}^{\Y})$ for each $t\in[0,T]$. We now prove that $\Lambda\coloneqq\lim_{N\rightarrow +\infty}\Lambda^{N}$ is a Lagrangian solution. Let us now consider $\dot{y}^{N}(t)= \hat{b}_{\Lambda^{N}}(y(t))-\bar{W}(t)$, and $\dot{y}(t)= \hat{b}_{\Lambda}(y(t))-\bar{W}(t)$ starting from some $\bar{y}\in B_{r}^{\Y}$, and where 
 \begin{align*}
     \overline{W}(t)\coloneqq 
\begin{pmatrix}
&\hskip -0,3cm -\sqrt{2\sigma}\frac{\de}{\de t}\overline{B}(t)\\
&\hskip -0,3cm 0
\end{pmatrix},
 \end{align*}
for some suitable Brownian motion $\overline{B}$. Notice that 

\begin{align*}
\norma{y^{N}(t)-y(t)}_{\Y}&\leq \int_{0}^{t}\norma{b_{\Lambda_{s}^{N}}(y^{N}(s))-b_{\Lambda_{s}}(y(s))}_{\Y}\de s\\
&\leq\int_{0}^{t}\norma{b_{\Lambda_{s}}(y^{N}(s))-b_{\Lambda_{s}}(y(s))}_{\Y}\de s+\\
&\phantom{formula}+\int_{0}^{t}\norma{b_{\Lambda_{s}^{N}}(y^{N}(s))-b_{\Lambda_{s}}(y^{N}(s))}_{\Y}\de s\\
&\leq L_{R}\int_{0}^{t}\norma{y^{N}(s)-y(s)}_{\Y}\de s + L_{R}\int_{0}^{t}\W_{1}(\Lambda_{s}^{N},\Lambda_{s})\de s.
\end{align*}
Then by Lemma \ref{lemma:estab}, there exists a positive constant $C$ such that
\begin{align*}
\norma{y^{N}(t)-y(t)}_{\Y} \leq   L_{R}\int_{0}^{t}\norma{y^{N}(s)-y(s)}_{\Y}\de s + L_{R}C\int_{0}^{t}\exp(sL_{R})\W_{1}(\overline{\Lambda}^{N},\overline{\Lambda})\de s,  
\end{align*}
and thus

\begin{align*}
 \norma{y^{N}(t)-y(t)}_{\Y}\leq C\exp(tL_{R})&\W_{1}(\overline{\Lambda}^{N},\overline{\Lambda}) +\\
 &+L_{R}\W_{1}(\overline{\Lambda}^{N},\overline{\Lambda})\exp(tL_{R})\int_{0}^{t}\exp(L_{R}(t-s))\de s\\
&\leq \exp(2tL_{R})(C+1)\W_{1}(\overline{\Lambda}^{N},\overline{\Lambda})
\end{align*}
and from here, we get by a comparison argument that
\begin{align*}
\norma{\overline{F}_{\Lambda^{N}}(t,0,\overline{y})-\overline{F}_{\Lambda}(t,0,\overline{y})}\leq \exp(2tL_{R})(C+1)\W_{1}(\overline{\Lambda}^{N},\overline{\Lambda}),
\end{align*}
and thus
\begin{align*}
\Lambda_{t}^{N}=(\overline{F}_{\Lambda^{N}}(t,0,\overline{y}))_{\#}\overline{\Lambda}^{N} \rightarrow  (\overline{F}_{\Lambda}(t,0,\overline{y}))_{\#}\overline{\Lambda}. 
\end{align*}
Hence by the uniqueness of Lagrangian solutions we get that $(\overline{F}_{\Lambda}(t,0,\overline{y}))_{\#}\overline{\Lambda}=\Lambda_{t}$. By Theorem \ref{thm:aux}, we conclude that $\Lambda$ is also a Eulerian solution, and we are done.
\end{proof}

\section{A mean-field model for a large network of interacting neurons}
\label{section4}

In this part, we consider a mean-field approach to study networks of spiking
neurons.  The time evolution, between two consecutive spikes, of the membrane potential $X^i(t)$ of a neuron $i$ embedded in a network, is described by a system of type \eqref{system2}. This model is of integrate-and-fire type, since the drift part accounts ({\it integrate}) for the inputs the neuron receives from its neighboring neurons. 
At each spiking time $\bar t$ ({\it fire}), that is when $X^i(\bar t)=X_F$ (firing threshold), the potential is reset to $X_R$ (resting potential) and all other neurons connected receive an additional potential depending on $\lambda^i$.
These inputs affect the membrane potential through the field $v$ that depends both on the actual state $X^i$ and on the coupling strength with the connected neurons. The dependence of $v$ on the state $X^i$ follows the principles of integrate-and-fire neuronal models with reversal potentials (see \cite{lansky1987diffusion}). The dependence of $v$ on $\lambda^i$, on the other hand, accounts for the different synaptic connection among neurons and could also include both inhibitory and excitatory contributions. Finally the dependence of the whole dynamics on the empirical measure $\Lambda$, which in this case represents the distribution of membrane potential values, can be used to model phenomena such as collective behavior, bursting, and avalanches, to name a few.

More precisely, we consider the system
\begin{align}\label{synap1}
\begin{dcases}
&X^{i}(t)=X_{0}^{i}+ \int_{0}^{t}v_{\Lambda_{s}^{N}}(X^{i}(s),\lambda^{i}(s))\de s+ \sqrt{2\sigma}B_{t}^{i},  \hskip 0,1cm \text{if $X^{i}(t^{-})<X_{F}$,}\\
&X^{i}(t)=X_{R}, \phantom{+\int_{0}^{t}v_{\Lambda_{s}^{N}}(X^{i}(s),\lambda^{i}(s))\de s+ \sqrt{2\sigma}B_{t}^{i}}\hskip 0,2cm \text{if $X^{i}(t^{-})=X_{F}$,}\\
&\lambda^{i}(t)=\lambda_{0}^{i}+\int_{0}^{t}\T_{\Lambda_{s}^{N}}(X^{i}(s),\lambda^{i}(s))\de s.
\end{dcases}
\end{align}

This model includes the feature that the interactions between the neurons are also subject to random synaptic weights \cite{delarue2019,faugeras2020asymptotic}.
Different scenarios can arise according to the form of $\T$.

In what follows, let us consider a function $\alpha:U\times \Y\rightarrow \R$, and define $\T:\Y\times \P_{1}(\Y) \rightarrow \F(U)$ as
\begin{align}\label{sec:newcho}
\begin{aligned}
&\T(y,\Psi)\coloneqq \int_{\Y}\alpha(\cdot,z)\de \Psi(z), \hskip 0,1cm \text{ such that} \hskip 0,1cm \int_{U}\int_{\Y}\alpha(u,z)\de \Psi(z)\de\mu(u)=0,\\
&\mu(u)=\sum_{k\geq 1}\beta_{k}\delta_{u_{k}}, \hskip0,1cm u_{k}\in U, \hskip 0,1cm \beta_{k}\in \R,
\end{aligned}
\end{align}
where $\mu\in \M(U)$ defines a signed measure. Furthermore, we consider a Lipschitz velocity field $b: \R^{d}\rightarrow \R^{d}$, and we define $v_{\T}:\Y\rightarrow \R^{d}$, $(x,\lambda)\mapsto b(x)$. Since the third equation of model \eqref{synap1} serves as a recipient to model the strength of the connection between neurons, we can think of $U$ as a discrete set that counts the number of neurons present in the brain. For the sake of simplicity, we may take $U=\Z^{d}$, endowed with a suitable metric that makes it compact.. By taking $\Psi =\Lambda_{t}^{N}$, $N\in\N$, we have that $\T(y,\Lambda_{t}^{N})$ in \eqref{sec:newcho} can be written as
\begin{equation}   
\label{tau_rsw}
\T(y,\Psi)=\frac{1}{N}\sum_{k\geq 1}\sum_{i=1}^{N}\beta_{k}\alpha(u_{k},X^{i}(t),\lambda^{i}(t)).
\end{equation}

 \begin{remark}
Expression \eqref{tau_rsw} 
 is of great generality and contains, as special cases, the random synaptic weights seen as constants, Bernoulli or Gaussian random variables  considered in previous works \cite{de2015hydrodynamic, delarue2019, eva_toymodel, faugeras2009constructive}.
 To support this claim, we can compare system \eqref{synap1} and \eqref{tau_rsw} with, for instance, the equation presented in \cite{delarue2019}
     \begin{equation} X^i(t)=X^i_0+\int_0^tb(X^i(s))ds+\sum_{j=1}^N\sum_{k\geq 1}J_N^{j\rightarrow i}1_{\{\tau_k^j\leq t\}}+W_t
 \end{equation}
 that combines the two equations for $X$ and $\lambda$ and where $\tau_k^j$ denotes the $k$-th spike time of  neuron $j$ and the synaptic weights $J_N^{j\rightarrow i}$ are defined in two representative cases:
 \begin{itemize}
 \item $J_N^{j\rightarrow i}=\alpha/N$, with $\alpha$ constant
 \item $J_N^{j\rightarrow i}=\beta\alpha^{i,j}/N$, with $\beta$ constant and $\alpha^{i,j} \sim Bern(p)$.
 \end{itemize}
 \end{remark}

Unlike \eqref{system2}, the dynamics of model \eqref{synap1} exhibit resets (jumps) corresponding to the spikes. Therefore, comments on the well-posedness of \eqref{synap1} are necessary and are provided in the following theorem.
In what follows, we state the main result regarding the existence and uniqueness of solutions for the system \eqref{synap1}. The hypotheses are essentially the same as those in the previous sections; however, we present the next result under the assumption that the fields $v,\T$ are globally Lipschitz.

\begin{theorem}\label{thm:neuron1}
Let us fix a filtered probability space $(\Omega,\B,\cP)$ endowed with a complete filtration $(\mathcal{F}_{t})_{t\in [0,T]}$ generated by the Brownian motion $(\bB(t))_{t\in[0,T]}$. Assume that $v_{\Psi}:\Y\rightarrow \R^{d}$, $(x,\lambda)\mapsto b(x)$, where $b:\R^{d}\rightarrow \R^{d}$ is a Lipschitz velocity field with Lipschitz constant $L_{b}$. Assume that $\T$ is defined as in \eqref{sec:newcho} where $\alpha$ is a measurable function with respect to $\B(U)\times \B(\Y)$, and it is Lipschitz in the second component with Lipschitz constant $L_{\alpha}$.  Furthermore, suppose that $X_{F}$ is a uniformly bounded random variable, that is $\sup_{\omega\in \Omega}X_{F}^{\omega}<+\infty$, and that  for all $i\in \N$, $(X_{R}, \Lambda^{i}(0))\in L^{2}(\Omega,\B,\P; \cE)$. Then the system \eqref{synap1} admits a unique strong solution.
\end{theorem}

\begin{proof}
In what follows, we use the same notation, and follow the argument used during the proof of Theorem \ref{importantprop}. Since by hypothesis  $\cP(X_{F}=+\infty)=0$, then we can take $\omega\in \Omega$ such that $X_{F}^{\omega}<+\infty$. We then propose the following Picard iteration. For $n\in\N\backslash \{0\}$,  with an abuse of notation, we set

\begin{align*}
\begin{dcases}
&\X_{n+1}(t)=\X_{0}+ \int_{0}^{t}\v_{\Lambda_{s}}^{N}(\X_{n}(s),\blambda_{n}(s))\de s+ \sqrt{2\sigma}\bB(t),\hskip 0,1cm \text{if $\X_{n}(t^{-})<\X_{F}$,}\\
&\X_{n+1}(t)=\X_{R}+ \int_{0}^{t}\v_{\Lambda_{s}}^{N}(\X_{R}(s),\blambda_{n}(s))\de s+ \sqrt{2\sigma}\bB(t),\hskip 0,1cm \text{if $\X_{n}(t^{-})=\X_{F}$,}\\
&\blambda_{n+1}(t)=e^{-\frac{t}{\theta}}\blambda_{0}+ \frac{1}{\theta}\int_{0}^{t}e^{\frac{s-t}{\theta}}g_{\Lambda_{s}}(\X_{n}(s),\blambda_{n}(s))\de s,\\
&\X_{0}\coloneqq \overline{\X}_{0},\\
&\blambda_{0}\coloneqq \overline{\blambda}_{0},
\end{dcases}
\end{align*}
where $\X_{R}\coloneqq \underset{N-times}{\underbrace{(X_{R},\ldots, X_{R})}}$, and $\X_{F}\coloneqq \underset{N-times}{\underbrace{(X_{F},\ldots, X_{F})}}$, and we have borrowed the notation used in \eqref{Picard1}. Notice that we have indicated $\X_{n}(t^{-})<\X_{F}$ meaning that each component satisfies such an inequality.  We then proceed as in proof of Theorem \ref{importantprop}. Nevertheless, we need to consider two cases. First, we need to consider when ${\bf X}_{n}(t^{-})<X_{F}$, for fixed $n\in\N$. In that cases, we need to show that there exists a constant $R>0$ such that
\begin{align}\label{new:dis1}
\E\left(\sup_{s\in [0,t]}\norma{\by_{n+1}(s)- \by_{n}(s))}^{2}\right)\leq \frac{(Rt)^{n}}{n!}
\end{align}
for all $n\in \N$, and all $t\in [0,T]$. Second, suppose that ${\bf X}_{n}(t^{-})=X_{F}$. Hence, we possibly find a further positive constant $\hat{R}$ such that \eqref{new:dis1} holds true. Then, we optimize between $R$ and $\hat{R}$ to find a new constant $R$ for which \eqref{new:dis1} holds true. From here, we may proceed with the proof of our desired result by following the same reasoning as in the proof of Theorem \ref{importantprop}.
\end{proof}

\begin{remark}
    We proved Theorem \ref{thm:neuron1} within a general framework in which 
$X_F$  is a bounded random variable. However, in most neuronal models, the spiking threshold is a physiological parameter and is typically treated as a constant.
\end{remark}

Having shown the existence of the solution, let us now show two pictures to illustrate the qualitative behavior of  \eqref{synap1}  in a simple case. In Figure \ref{fig:X_N_100}, we observe the dynamics of $X^{i}(t)$ in the one-dimensional case for $N=1000$, and where we have chosen $X_{F}=0.7$, $X_{R}=0.01$. To provide insights into the behavior of \eqref{synap1}, we considered a linear system of the form $v_{\Lambda_{t}^{N}}(X^{i}(t),\lambda^{i}(t))= a\cdot X^{i}(t)+ b\cdot\lambda^{i}(t) + c\cdot \Lambda_{t}^{i}$, $\T_{\Lambda_{t}^{N}}(X^{i}(t),\lambda^{i}(t)) = d\cdot X^{i}(t) + e\cdot\lambda^{i}(t) + f\cdot \Lambda_{t}^{N}-\E(d\cdot X^{i}(t) + e\cdot\lambda^{i}(t) + f\cdot \Lambda_{t}^{N})$, with $a = 0.5$, $b = 0.3$, $c = 0.2$; $d = 0.4$, $e = 0.2$, $f = 0.1$ (we omit the units since this is just an illustrative example). Notice that whenever $X^{i}(t)=X_{R}$, $\lambda^{i}(t)$ exhibits singularities in its trajectories.

\begin{figure}[htbp!]
    \centering
    \begin{subfigure}{0.45\textwidth}
        \includegraphics[width=\linewidth]{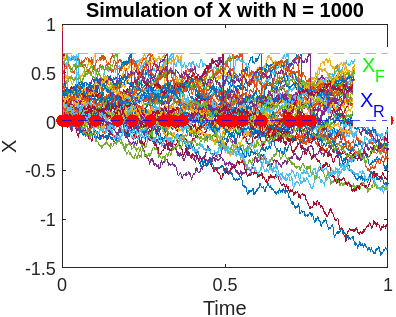}
        \caption{Trajectories of $X^{i}$ with N = 1000}
    \end{subfigure}
    \hspace{0.5cm}
    \begin{subfigure}{0.45\textwidth}
        \includegraphics[width=\linewidth]{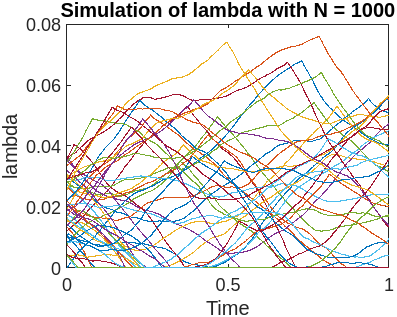}
        \caption{Trajectories of $\lambda^{i}$ with N = 1000}
    \end{subfigure}
    \caption{Trajectories for $X$, and $\lambda$ with $N=1000$. }
    \label{fig:X_N_100}
\end{figure}


\subsection{Heterogeneous population of interacting neurons}\label{sec:evol1}
In this section, we consider a more general version of the model studied in the previous sections. In particular, we include in the second equation of \eqref{system1} an additional noise.

This network can be interpreted as a collection of neurons with different characteristics (resulting for instance in a different firing rate)
and therefore, it is a model of an heterogeneous
network. The weights distributions in cortex have been observed to be broad and neurons can belong to different subgroups with different behaviors. Moreover, the variance of the temporal fluctuations of the input currents may be large enough to endow the network with its own source of variability (Section 6.4.2 in \cite{la2021mean}).

For these reasons we consider the system 

\begin{align}\label{evol:system1}
\begin{dcases}
&X^{i}(t)=X_{0}^{i}+ \int_{0}^{t}v_{\Lambda_{s}^{N}}(X^{i}(s),\lambda^{i}(s))\de s+ \sqrt{2\sigma}B^{i}(t),\\
&\lambda^{i}(t)=\lambda_{0}^{i}+\int_{0}^{t}\T_{\Lambda_{s}^{N}}\left(X^{i}(s),\lambda^{i}(s)+ R(s)\right)\de s ,\\
&R(t)\coloneqq \sum_{h\in \N}a_{h}W_{h}(t)e_{h},
\end{dcases}
\end{align}
where $(W_{h}(t))_{h\in \N}$ is a sequence of independent Brownian motions, which are also independent of $B^{i}$ for all $i\in \N$. Here, $(a_{h})_{h\in\N}$  is a sequence of non-negative numbers such that $\sum_{h}a_{h}^{2}<+\infty$, and $(e_{h})_{h}$ is a sequence of signed measures supported in $U$ such that $\int_{U}\de e_{h}(u)=0$ for all $h\in\N$. The existence of strong solutions can be treated by following the same approach of Theorem \ref{importantprop}. Indeed, we have the following result.

\begin{proposition}
Let us fix a filtered probability space $(\Omega,\B,\cP)$ endowed with a complete filtration $(\mathcal{F}_{t})_{t\in [0,T]}$ generated by the Brownian motion $(B(t))_{t\in[0,T]}$. Assume that for every $y\in \Y$ and $\Psi\in \P_{1}(\Y)$ the velocity field $v_{\Psi}:\Y\rightarrow \R^{d}$ satisfies (A\ref{A1})-(A\ref{A3}), and the operator $\T_{\Psi}$ as defined in \eqref{sec:newcho} satisfies (B\ref{B1})-(B\ref{B4}). Moreover, let us consider $(\lambda_{h})_{h\in\N}$ be a sequence of non-negative numbers such that $\sum_{h}\lambda_{h}^{2}<+\infty$, and let $(W_{h}(t))_{h\in \N}$, $t\in [0,T]$ a sequence of independent real-valued Brownian motions, such that for all $t\in [0,T]$, $R(t)\in \M_{0}(U)$. Then, for every choice $Y_{0}\coloneqq(\overline{X}_{0}, \overline{\lambda}_{0})$ as initial condition of \eqref{evol:system1} such that $Y_{0}\in L^{2}(\hat{\Y})$, the system \eqref{evol:system1} has a unique solution $Y$ that belongs to $\M_{{\rm ad}}^{2}(\hat{\Y})$.
\end{proposition}
\begin{proof}
Notice that for almost all $\omega\in\Omega$, $\T_{\Lambda_{t}^{N}}(X^{i}(t),\lambda^{i}(t)+ R(t))$ satisfies our assumptions (B\ref{B1})-(B\ref{B4}). Therefore, we may apply Theorem \ref{importantprop} by replacing $\T_{\Lambda_{t}^{N}}$ with $\overline{\T}_{\Lambda_{t}^{N}}(X^{i}(t),\lambda^{i}(t))\coloneqq \T_{\Lambda_{t}^{N}}(X^{i}(t),\lambda^{i}(t)+ R(t))$, and where \eqref{evol:system1} can be written as

\begin{align*}
\begin{dcases}
&X^{i}(t)=X_{0}^{i}+ \int_{0}^{t}v_{\Lambda_{s}^{N}}(X^{i}(s),\lambda^{i}(s))\de s+ \sqrt{2\sigma}B^{i}(t),\\
&\lambda^{i}(t)=\lambda_{0}^{i}+\int_{0}^{t}\overline{\T}_{\Lambda_{t}^{N}}(X^{i}(s),\lambda^{i}(s))\de s.\\
\end{dcases}
\end{align*}
Moreover, in order to apply the Picard iteration used in Theorem \ref{importantprop}, we define $\overline{g}_{\Lambda_{t}}(X_{n}^{i}(t),\lambda_{n}^{i}(t))\coloneqq \lambda_{n}^{i}(t)+ \theta \overline{\T}_{\Lambda_{t}^{N}}(X_{n}^{i}(t),\lambda_{n}^{i}(t))$ for some suitable $\theta>0$. Then, we consider the Picard iteration \eqref{Picard1} with $g_{\Lambda_{t}}$ replaced by $\overline{g}_{\Lambda_{t}}$. That is, we have

\begin{align*}
\begin{dcases}
&X_{n+1}^{i}(t)=X_{0}^{i}+ \int_{0}^{t}\v_{\Lambda_{s}}^{N}(X_{n}^{i}(s),\lambda_{n}^{i}(s))\de s+ \sqrt{2\sigma}B^{i}(t),\\
&\lambda_{n+1}^{i}(t)=e^{-\frac{t}{\theta}}\lambda_{0}+ \frac{1}{\theta}\int_{0}^{t}e^{\frac{s-t}{\theta}}\overline{g}_{\Lambda_{s}}(X_{n}^{i}(s),\lambda_{n}^{i}(s))\de s,\\
\end{dcases}
\end{align*}
and thus we proceed by following the reasoning of Theorem \ref{importantprop}, and our conclusion follows.
\end{proof}

\section{Possible future directions}
Notice that for equations of the form \eqref{system2}, we have assumed that hypothesis C1 holds true to derive the existence and uniqueness of Lagrangian solutions. In the case of equations of the form \eqref{synap1}, we are unable to formulate such a condition in a manner that takes into account the effect of the random variable $X_{F}$, which controls the spike times of $X$. However, since simulations show that the trajectories of $\lambda$ are $C^{0}$ and the trajectories of $X$ are almost the same as those we can simulate through \eqref{system1}, we conjecture that it is possible to obtain the existence of both Lagrangian and Eulerian solutions. Nonetheless, this remains a point for further investigation. Additionally, we would like to reformulate all the results by relaxing assumption (C1). 
Moreover, we seek to analyze the asymptotic behavior of a system of particles as time diverges. More precisely, we are interested in whether the continuity equation governed by the operator in assumption (C1) admits a stationary solution. In this aspect we intend to compare our analysis with results such as those in \cite{pakdaman2009dynamics}. Specifically, we aim to determine a decay rate for the family of probability measures ${\Lambda_{t}^{N}}$ using a Doeblin-type condition; see, for instance, \cite[Theorem 2.3]{canizo2019asymptotic}.

\appendix
\section{Well-posedness of ODEs in Banach spaces}
Let us recall the following Theorem by Brezis \cite[Theorem 1.4]{Brezis1973} about the well-posedness of ODEs in Banach spaces.

\begin{theorem}[\cite{Brezis1973}]\label{Thm:Brezis}
Let $(E,\norma{\cdot}_{E})$ be a Banach space, $C$ a closed convex subset of $E$, and let $A(t,\cdot):C\rightarrow E$, $t\in[0,T]$, be a family  of operators satisfying the following properties:

\begin{enumerate}
\item[(i)] there exists a constant $L\geq 0$ such that for every $c_{1},c_{2}\in C$ and $t\in[0,T]$
\begin{align*}
\norma{A(t,c_{1})-A(t,c_{2})}_{E}\leq L\norma{c_{1}-c_{2}}_{E};
\end{align*}
\item[(ii)] for every $c\in C$ the map $t\mapsto A(t,c)$ is continuous in [0,T];
\item[(iii)] for every $R>0$ there exists $\theta>0$ such that

\begin{align*}
c\in C,\hskip0,1cm \norma{c}_{E}\leq R\Rightarrow \hskip0,1cm c+\theta A(t,c)\in C.
\end{align*}
\end{enumerate}
Then for every $\bar{c}\in C$ there exists a unique curve $c:[0,T]\rightarrow C$ of class $C^{1}$ satisfying $c_{t}\in C$ for all $t\in [0,T]$ and

\begin{align*}
\frac{\de}{\de t}c_{t}=A(t,c_{t}) \hskip0,1cm \text{in $[0,T]$, $c_{0}=\bar{c}$.}
\end{align*}
Moreover, if $c^{1},c^{2}$ are the solutions starting from the initial data $\bar{c}^{1}, \bar{c}^{2}\in C$, respectively, we have

\begin{align*}
\norma{c_{t}^{1}-c_{t}^{2}}\leq {\rm e}^{Lt}\norma{\bar{c}^{1}-\bar{c}^{2}}_{E}\hskip 0,1cm \text{for every $t\in [0,T]$.}
\end{align*}
\end{theorem}

In what follows, we consider the following generalization of the previous result proved in \cite[Corollary 2.3]{MS2020}.

\begin{corollary}[{\cite[Corollary 2.3]{MS2020}}]\label{Cor:MS}
Let hypotheses $(ii)$ and $(iii)$ of Theorem \ref{Thm:Brezis} hold for a family of operators $A(t,\cdot): C\rightarrow E$, $t\in [0,T]$. Assume in addition, that

\begin{enumerate}
\item[($i^\prime$)] for every $R>0$ there exists a constant $L_{R}\geq 0$ such that for every $c_{1},c_{2}\in C\cap B_{R}$ and $t\in[0,T]$,

\begin{align*}
\norma{A(t,c_{1})-A(t,c_{2})}_{E}\leq L_{R}\norma{c_{1}-c_{2}}_{E};
\end{align*}
\item[($ii^\prime$)] there exists $M>0$ such that for every $c\in C$, there holds

\begin{align*}
\norma{A(t,c)}_{E}\leq M(1+\norma{c}_{E}).
\end{align*}
Then for every $\bar{c}\in C$ there exists a unique curve $c:[0,T]\rightarrow C$ of class $C^{1}$ satisfying $c_{t}\in C$ for all $t\in [0,T]$, and
\begin{align*}
\frac{\de}{\de t}c_{t}=A(t,c_{t}) \hskip 0,1cm \text{in $[0,T]$,  $c_{0}=\bar{c}$.}
\end{align*}
Moreover, if $c^{1}, c^{2}$ are the solutions starting from the initial data $\bar{c}^{1}, \bar{c}^{2}\in C\cap B_{R}$, respectively, there exists a constant $L=L(M,R,T)>0$ such that

\begin{align*}
\norma{c_{t}^{1}-c_{t}^{2}}\leq {\rm e}^{Lt}\norma{\bar{c}^{1}-\bar{c}^{2}}_{E}\hskip 0,1cm \text{for every $t\in [0,T]$.}
\end{align*}

\end{enumerate}
\end{corollary}

\section{A Kolmogorov equation for measures}

Here, suppose that $X$ is a separable Banach space with norm $\norma{\cdot}$, and let $\B(X)$ be its Borel $\sigma$-algebra. We denote by $C_{b}(X)$, the Banach space of all uniformly continuous and bounded functions $f:X\rightarrow \R$, endowed with the supremum  norm $\norma{\cdot}_{0}$. Furthermore, we consider a strongly continuous semigroup of linear operators $\{P_{t}\}_{t\geq 0}\subset \mathcal B(C_{b}(X))$ with the Markovian property, that is, there exists a family $\{\pi_{t}(x,\cdot): t\geq 0, x\in X\}$ of probability measures on $X$ such that

\begin{itemize}
\item the map $\R_{+}\times X\rightarrow [0,1]$, $(t,x)\mapsto \pi_{t}(x,\Gamma)$ is measurable, for any Borel set $\B(X)$;
\item $\pi_{t+s}(x,\Gamma)=\int_{X}\pi_{s}(y,\Gamma)\pi_{t}(x,\de y)$ for all $t,s\geq 0$, $x\in X$, and $\Gamma\in \B(X)$;
\item for any $x\in X$, $\pi_{0}(x,\cdot)=\delta_{x}(\cdot)$, the probability measure concentrated in $x$;
\item $P_{t}\phi(x)=\int_{X}\phi(y)\pi_{t}(x,\de y)$, for any $t\geq 0$, $\phi\in C_{b}(X)$, $x\in X$;
\item for any $\phi\in C_{b}(X)$, $x\in X$, the function $\R_{+}\rightarrow \R$, $t\mapsto P_{t}\phi(x)$ is continuous.
\end{itemize}

Then we define the linear operator $(\mathcal L,D(\mathcal L))$ where

\begin{align}\label{gen1}
\begin{aligned}
&D(\mathcal L)\coloneqq \left\{ u\in X: \lim_{t\downarrow 0}\frac{P_{t}u-u}{t}\hskip 0,1cm \text{exists}\right\},\\
&\mathcal Lu\coloneqq \lim_{t\downarrow 0}\frac{P_{t}u-u}{t}\hskip 0,1cm u\in D(\mathcal L).
\end{aligned}
\end{align}

\begin{definition}\label{defweak}
Let $\overline{\mu}\in \M_{b}(X)$. We say that a family of measures $\{\mu_{t}\}_{t\geq 0}$ is a solution of the measure equation
\begin{align}\label{measeq}
\begin{cases}
&\frac{\de}{\de t}\int_{X}\phi \mu_{t}(\de x)=\int_{X}\mathcal L\phi(x)\mu_{t}(\de x), \hskip 0,1cm t\geq 0, \phi\in D(\mathcal L),\\
&\mu_{0}=\overline{\mu}, \hskip 0,5cm \mu\in \M_{b}(X),
\end{cases}
\end{align}
if the following conditions are fulfilled:
\begin{itemize}
\item the total variation of the measure $\mu_{t}$ satisfies
\begin{align*}
\int_{0}^{T}\norma{\mu_{t}}_{TV}\de t<+\infty, \hskip 0,3cm T>0;
\end{align*}
\item for any $\phi\in D(\mathcal L)$, the real valued function $t\mapsto \int_{X}\phi(x)\mu_{t}(\de x)$ is absolutely continuous, and  for any $t\geq 0$ it holds
\begin{align*}
\int_{X}\phi(x)\mu_{t}(\de x)-\int_{X}\phi(x)\overline{\mu}(\de x)=\int_{0}^{t}\int_{X}\mathcal L\phi(x)\mu_{s}(\de x)\de s.
\end{align*}
\end{itemize}
\end{definition}
Let us now recall an important result about the well-posedness of \eqref{measeq} when $X$ is a Hilbert space, \cite[Theorem 1.2]{MR2487956}.

\begin{theorem}\label{wellpod}
Let $\{P_{t}\}_{t\geq 0}$ be a strongly continuous Markov semigroup and let $(\mathcal L,D(\mathcal L))$ be its infinitesimal generator, defined as in \eqref{gen1}. Then the formula
\begin{align*}
\nm{\phi, P_{t}^{\ast}F}_{\mathcal B(C_{b}(X),(C_{b}(X))^{\ast})}\coloneqq \nm{F, P_{t}\phi}_{\mathcal B(C_{b}(X),(C_{b}(X))^{\ast})}
\end{align*}
defines a semigroup $(P_{t}^{\ast})_{t\geq 0}$ of linear and continuous operators on $(C_{b}(X))^{\ast}$ that maps $\M_{b}(X)$ into $\M_{b}(X)$. Moreover, for any $\overline{\mu}\in \M_{b}(X)$, $\phi\in C_{b}(X)$ the map $t\mapsto \int_{X}\phi(x)P_{t}^{\ast}\mu(\de x)$ is continuous, and if $\phi \in D(\mathcal L)$ then it is also differentiable with continuous differential 

\begin{align}\label{newdiffeq}
\frac{\de}{\de t}\int_{X}\phi(x)P_{t}^{\ast}\mu(\de x)= \int_{X}\mathcal L\phi(x) P_{t}^{\ast} \mu(\de x).
\end{align}
Finally, for any $\mu\in \M_{b}(X)$ there exists a unique solution of the measure equation \eqref{measeq} given by $\{P_{t}^{\ast}\mu\}_{t\geq 0}$.
\end{theorem}
\begin{proof}
The proof of this theorem can be performed by using the same argument as the one in \cite[Theorem 1.2]{MR2487956}. The only difference is that we consider a separable Banach space rather than a Hilbert space.
\end{proof}

\section{Aldous tightness criterion}\label{Aldoscrit1}
Let $E$ be a Polish space, and for each $m\in \N$, we denote by $\mathcal{D}_{E}^{m}=\mathcal{D}_{E}[0,m]$ the Skorokhod space. That is, the space of all càdlag trajectories defined on $[0,m]$ with values in $E$. In the same way, we denote by $\mathcal{D}_{E}^{\infty}$ the Skorokhod space $\mathcal{D}_{E}[0,+\infty)$. Let $X_{n}$ be random variables of $\mathcal{D}_{E}^{\infty}$.

\begin{theorem}[{\cite[Theorem 16.10]{billi1999}}]
Suppose that 

\begin{enumerate}
    \item \begin{align*}
\lim_{a\rightarrow +\infty}\limsup_{n}\cP(\norma{X_{n}}_{\mathcal{D}_{E}^{m}}\geq a)=0,    
\end{align*}
holds for all $m\in\N$.

\item For each $\varepsilon>0,\eta>0,m$ there exists $\delta_{0}>0$ and $N_{0}\in \N$ such that, if $\delta\leq \delta_{0}$, and $N\geq N_{0}$, and if $\tau$  is a discrete $X_{n}$-stopping time satisfying $\tau \leq m$, then
\begin{align*}
 \cP\left(\vert X_{n}(\tau+\delta)-X_{n}(\tau)\vert \geq \varepsilon\right)\leq \eta.
\end{align*}
Then $X_{n}$ is tight.
\end{enumerate}

\end{theorem}

\section*{Acknowledgments}
The authors would like to thank Alessandro Baldi and Marco Morandotti for inspiring discussions that greatly influenced this work. 
The author G.D. acknowledges financial support under the MIUR-PRIN 2022 project “Non-Markovian dynamics and non-local equations”,  No. 202277N5H9 - CUP: D53D23005670006.
G.D. participates also in the INdAM - GNAMPA Project, \textit{Deterministic Control
of Stochastic Dynamics} CUP E53C23001670001.
AMH has been supported by project PRIN 2022 ``understanding the LEarning process of QUantum Neural networks (LeQun)'', proposal code 2022WHZ5XH -- CUP J53D23003890006.

\bibliographystyle{plain}

 \end{document}